\documentclass[11pt]{article}

\usepackage{graphicx}
\usepackage{latexsym,amsmath,amsfonts,amscd, amsthm, dsfont}
\usepackage{mathtools}
\usepackage{bm,color}
\usepackage{epsfig,verbatim,epstopdf,graphics}
\usepackage{changebar}
\usepackage{multirow}
\usepackage{graphics} 
\usepackage{caption}
\usepackage{placeins}
\usepackage{ragged2e}
\usepackage{subfig}
\usepackage{xcolor}
\usepackage[pdftex]{hyperref}
\hypersetup{
    colorlinks = false,
    linkcolor =black ,
    linkbordercolor=white,
    citecolor=black,
    citebordercolor=white,
    urlbordercolor =white,
}
\usepackage{epstopdf}
\usepackage{nicefrac}
\usepackage[utf8]{inputenc}
\usepackage[english]{babel}
\usepackage{csquotes}
\usepackage{comment}
\usepackage{algorithmic}
\usepackage{url}    
\usepackage{yhmath}
\usepackage{booktabs} 
\usepackage{tikz}
\usepackage{verbatim}
\usetikzlibrary{arrows,backgrounds,snakes,shapes}
\numberwithin{equation}{section}

\usepackage{xcolor}   

\graphicspath{{./}{./figure/}}
\allowdisplaybreaks

\newtheoremstyle{plainNoItalics}{}{}{\normalfont}{}{\bfseries}{.}{ }{}

\theoremstyle{plain}
\newtheorem{thm}{Theorem}[section]

\theoremstyle{plainNoItalics}
\newtheorem{cor}[thm]{Corollary}
\newtheorem{lem}[thm]{Lemma}

\newtheorem{remark}[thm]{Remark}

\newcommand{\R}{\mathbb{R}}  
\newcommand\norm[1]{\left\lVert#1\right\rVert}

\def\Re{{\hbox{\rm I \hskip -.075in R}}}

\newcommand{\beq}{\begin{equation}}
\newcommand{\eeq}{\end{equation}}
\newcommand{\be}{\begin{eqnarray}}
\newcommand{\ee}{\end{eqnarray}}
\newcommand{\beno}{\begin{eqnarray*}}
	\newcommand{\eeno}{\end{eqnarray*}}

\def\cC{{\mathcal C}}
\def\cD{{\mathcal D}}

\def\cF{{\mathcal F}}
\def\cG{{\mathcal G}}
\def\cH{{\mathcal H}}
\def\cI{{\mathcal I}}
\def\cJ{{\mathcal J}}
\def\cL{{\mathcal L}}
\def\cM{{\mathcal M}}
\def\cN{{\mathcal N}}
\def\cO{{\mathcal O}}
\def\cP{{\mathcal P}}
\def\cR{{\mathcal R}}
\def\cT{{\mathcal T}}
\def\eps{\epsilon}
\def\vp{\mathbf{p}}
\def\tvp{\tilde{\mathbf{p}}}
\def\vq{\mathbf{p}}
\def\mB{{\mathrm B}}
\def\mbL{{\mathbb L}}
\def\omB{{\overline{\mathrm B}}}
\def\mrE{{\mathrm E}}
\def\mrH{{\mathrm H}}
\def\mrI{{\mathrm J}}
\def\mrL{{\mathrm L}}

\def\mn{{\mathbf n}}
\def\ophi{{\overline{\phi}}}

\def\tPi{{\tilde{\Pi}}}
\def\tw{{\tilde{w}}}
\def\tp{{\tilde{p}}}
\def\dell{\mathrm{d}}
\def\Texit{T_{\mathrm{exit}}}
\def\perpX{\perp} 
\makeatletter

\newcommand{\Rmnum}[1]{\expandafter\@slowromancap\romannumeral #1@}
\makeatother

\begin{document}
\baselineskip=1.5pc
\begin{center}
{\bf
Gradient Invariance of Slow Energy Descent: Spectral Renormalization and Energy Landscape Techniques.
}
\end{center}

\vspace{.2in}
\centerline{
Hayriye Guckir Cakir\footnote{
 Department of Mathematics, Michigan State University, East Lansing, MI, 48824. E-mail: guckirha@msu.edu
} and
Keith Promislow \footnote{Department of Mathematics, Michigan State University, East Lansing, MI, 48824. E-mail: promislo@msu.edu}
}

\bigskip
\noindent
{\bf Abstract.}
For gradient flows of energies, both spectral renormalization (SRN) and energy landscape (EL) techniques have been used 
to establish slow motion of orbits near low-energy manifold. We show that both methods
are applicable to flows induced by families of gradients and compare the scope and specificity of the results.  
The SRN techniques capture the flow in a thinner neighborhood of the manifold, affording a leading order representation of 
the slow flow via as projection of the flow onto the tangent plane of the manifold.  The SRN approach requires a spectral gap in
the linearization of the full gradient flow about the points on the low-energy manifold. We provide conditions on the choice of gradient
under which the spectral gap is preserved, and show that up to reparameterization the slow flow is invariant under these choices of gradients.
The EL methods estimate the magnitude of the slow flow, but cannot capture its leading order form. However the EL only requires
normal coercivity for the second variation of the energy, and does not require spectral conditions on the linearization of the full flow. 
It thus applies to a much larger class of gradients of a given energy. We develop conditions under which the assumptions of the 
SRN method imply the applicability of the EL method, and identify a large family of gradients for which the EL methods apply. In particular we apply both 
approaches to derive the interaction of multi-pulse solutions within the 1+1D Functionalized Cahn-Hilliard (FCH) gradient flow,
deriving gradient invariance for a class of gradients arising from powers of a homogeneous differential operator. 


\section{Introduction}

Gradient flows play a fundamental role in material science, biology, and other physical systems in which
dissipation is dominant. They provide mechanisms for self organization of patterns that minimize
the underlying energy of the system. The basic structure is provided by an energy
$\mrI$ that is a smooth map from a Hilbert space $H$ into $\R$, and a gradient, $\cG$, that relates the flow
of the system to the dissipation of the energy. Typically the energy is naturally posed in terms of the inner product
on a larger Hilbert space $X$ that lies between $H$ and its $X$-dual, $H'$. The underlying PDE takes the form
$$
u_t =-\cG \nabla_X \mrI(u),
$$
where $\nabla_X\mrI$ denotes the variational derivative of $\mrI$ in the $X$ inner product, and $\cG$ a non-negative, $X$-self-adjoint linear operator.
The energy $\mrI$ decreases along the orbits and minimizers of $\mrI$ are strong candidates for asymptotically stable equilibrium of the 
gradient flow. The energy landscape (EL) method arose to identify conditions under which manifolds of low-energy configurations engender 
slow flows that remain trapped within a thin neighborhood of the manifold. The EL method seems to have originated in the study of slow motion 
of radial interfaces in the Cahn-Hilliard system, \cite{ABF-98}, and was developed into a more general framework in \cite{OR-07} and more recently in
\cite{BFK-18}. The method makes few direct assumptions on the smoothness of the manifold nor upon the gradient, requiring only that the 
energy has little variation over the manifold, increases uniformly in the direction normal the manifold, and that there is well-defined projection from an $H$-neighborhood of the manifold onto the manifold. It is natural to  compare these results to the spectral renormalization (SRN) framework developed in \cite{Prom-02} for damped-forced  Hamiltonian systems and adapted in \cite{DKP-07} and \cite{BDKP-13} to singularly-perturbed reaction diffusion systems. 

The SRN method establishes the existence of slow flows in a neighborhood of a manifolds comprised of quasi-steady solutions. 
It makes detailed assumptions on the spectrum of the linearization, $\mbL:=-\cG\nabla_X^2 \mrI,$ of the vector field $F:=-\cG\nabla_X\mrI$ at 
the points on the manifold and renormalizes the estimates on the point-wise linearized operators into nonlinear semigroup estimates that 
contract the flow from a larger neighborhood into a substantially thinner neighborhood of the manifold.  The SRN method requires the manifold 
to be a graph over a finite dimensional space. Heuristically, if the vector field evaluated on the manifold satisfies 
$\|F(u)\|_H\approx \delta$ then the slow flow evolves on an $O(\delta)$ time-scale. The attracting neighborhood for the SRN approach has 
an $O(\delta)$ $H$-norm thickness and the distance of the orbit to the manifold contributes an $O(\delta^2)$ error. On the other hand the 
EL method embeds the manifold in a forward invariant neighborhood with an $O(\sqrt{\delta})$ thickness in the $H$-norm, whose $O(\delta)$
contribution to the error swamps the resolution of the slow flow. The SRN method resolves the leading order terms in the projection of 
the residual flow onto the tangent plane of the manifold, yielding a finite dimensional, closed form reduction of the slow flow. 
The EL approach affords bounds on the rate of the slow flow, but does not extract leading order information on the projection of the slow flow onto 
the tangent plane of the manifold. 

While the SRN method is quite general, applying to broad classes of damped-dispersive and dissipative systems, it requires significantly 
more machinery to apply than the EL approach, in particular it requires a spectral gap condition on the point-wise linearizations of the full gradient flow
at each location on the manifold.  For a given energy we establish conditions under which families of gradients which share the same kernel preserve 
the spectral gap.  We show that within these families the slow flows are equivalent up to reparameterization. To compare the applicability of the SRN and EL 
approaches, we develop mild additional conditions under which the assumptions of the SRN method guarantee the applicability of the EL approach. Indeed the 
generality of the EL approach allows it to encompass a substantially larger class of gradients than the SRN methodology. It is not intuitively obvious what 
becomes of the slow flow for choices of gradients for which the SRN fails while the EL approach holds. It is unclear if the failure of the SRN approach is technical, 
or if there is the potential for a more complex flow that is not slaved at leading order to its projection onto the tangent plane of the manifold. 
 
The EL approach has strong analogy to the much older orbital stability approach for Hamiltonian systems,
pioneered by  Brooke Benjamin, \cite{BB-72}. These exploit the conservation of the underlying energy, $\cH: H\mapsto \R$, rather than its decay, to 
maintain proximity of solutions of the Hamiltonian flow to a manifold of orbits.  The Hamiltonian flows take the form
$$
u_t = \cJ \nabla_X\cH,
$$
where the linear operator  $\cJ$ is skew with respect to 
the inner product of a Hilbert space $X$, which again resides between $H$ and its $X$-induced dual $H'$.
The approach characterizes critical points of the energy $\cH$ as minimizers subject to additional constraints induced by conserved 
quantities arising from symmetries of the energy.  The symmetries generate a manifold of equilibrium from the orbit of a single critical point under their 
group action. The orbital stability approach has broad applicability since it is largely independent of the specific form 
of the skew operator, and relies principally upon the analysis of the second variation of the energy $\cH$ at the point on the manifold of equilibrium.  
This is fortuitous as the second variation, $\nabla_X^2 \cH$ is a self-adjoint linear operator in the inner-product in which it is taken, while the 
full linearization, $\cJ\nabla_X^2 \cH$ is generically not self-adjoint. If the critical point of
$\cH$ is a strict minimizer, then the second variation has no negative eigenvalues; however this is rarely the case. Various stability indices have 
been developed that relate the number of negative eigenvalues of $\nabla_X^2 \cH$ to the number of complex eigenvalues of  
$\cJ\nabla_X^2 \cH$ with positive real part: eigenvalues which denote instability. Generically the larger the number of negative eigenvalues of the 
second variation, the greater the number of instabilities that are available to the flow.  A central result is that if the conserved 
quantities of the flow constrain it to lie in a finite co-dimensional space, then the relevant index is the number of negative eigenvalues of the second 
variation constrained to act on the reduced space. The calculation of this constrained eigenvalue count is the basis of the 
seminal work of Grillakis,  Shatah, and Strauss, \cite{GSS-I, GSS-II} and is summarized in \cite[Chapter 5]{KapProbook}. This constrained
eigenvalue count approach is  exploited in this work to establish the implication of the EL assumptions under the SRN hypotheses. Indeed, the SRN 
framework was originally derived to extend the orbital stability approach to classes of weakly damped-forced Hamiltonian systems arising in nonlinear 
optics.

As a test case, we apply both the SRN and EL approaches to the gradient flows of the Functionalized Cahn-Hilliard (FCH) free energy on a bounded, 
one-dimensional domain. The FCH free energy, presented in \cite{ProWetton} and in \cite{Gavish2011CurvatureDF}, is a reformulation of the energy of 
oil-water-surfactant microemulsions proposed
by \cite{TS-87} and revised in \cite{GS-90}. The FCH assigns an energy to a mixture of surfactant and solvent according to the volume fraction, $u$ of surfactant 
via its proximity to the large class of solutions of the second-order nonlinear system: 
\beq\label{e:NL-PDE}
\eps^2\Delta u= W'(u),
\eeq
subject to appropriate boundary conditions. More specifically the FCH energy takes the form 
\begin{equation}
\label{e:FCH}
\begin{aligned}
\cF\left(u\right)=\displaystyle\int_{\Omega}\frac{1}{2}\left(\varepsilon^2\Delta{u}-W'(u)\right)^2-\varepsilon^p\left(\frac{\eta_1\varepsilon^2}{2}\lvert\nabla{u}\rvert^2+\eta_2{W(u)}\right)dx, 
\end{aligned}
\end{equation}
where $\varepsilon\ll1$ is the ratio of amphiphilic molecule length to domain size and $\eta_1>0$, $\eta_2\in\R$. 
For $p=1$, the FCH corresponds to the strong functionalization while for $p=2$ it is a model for the weak functionalization. We assume that $W(u)$ is a double-well with two \emph{unequal} depth minima at $b_-<b_+$, satisfying $W(b_-)=0>W(b_+)$. The minima are non-degenerate in the sense that
$\alpha_\pm\coloneqq{W}''(b_\pm)>0$. 
As we restrict ourselves to one space dimension, the functionalization terms, those with the prefactors $\eta_1$ and $\eta_2$, 
play a negligible role and we set them equal to zero. In this case all solutions of the 1D version of (\ref{e:NL-PDE}) are global minimizers of the FCH free energy.  In \cite{ProZhang}, the existence of global minimizers was established over a variety of admissible function space for a class of generalizations of the FCH free energy.

\section{Spectral Renormalization and Energy Landscape approaches for Quasi-Steady Flows}
We present frameworks for the SRN  and the EL approaches for deriving slow 'quasi-steady' flows in neighborhoods of 
manifolds with low energy variation.
We consider classes of gradients with common kernels, and derive conditions on the gradients under which the SRN applies uniformly.
We also develop conditions under which the SRN assumptions satisfy the assumptions required to apply the EL approach, and show
that this includes choices of gradients for which the SRN does not directly apply.

\subsection{The Spectral Renormalization Framework}
The framework presented in \cite{Prom-02} was designed for damped-forced dispersive wave systems but applies more
generally to abstract dynamical system of the form
\beq
\label{e:abstract_ev}
u_t = F(u),
\eeq
that are locally well-posed on a pair of nested Hilbert spaces $H\subset X\subset H'$.  The key assumption is the 
existence of a quasi-steady manifold $\cM$ which is explicitly parameterized as the graph of a map $\Phi:\cP\subset \R^n\mapsto H$ 
\beq
 \cM = \{  \Phi(\vp) \, \bigl | \vp \in \cP\subset\R^n \}.
 \eeq
 The domain $\cP$ may be with or without boundary. We assume that the vector field $F$ admits an expansion of the 
 form
 \beq
 \label{e:F-taylor}
 F(\Phi+v) = \cR(\vp) + \mbL_\vp v + \cN_S(v),
 \eeq
where the residual, $\cR(\vp):=F(\Phi(\vp))$ is small,  $\mbL=\mbL_\vp$ is the linearization of $F$ 
at $\Phi(\vp)$ and the nonlinearity for the spectral approach satisfies a generic estimate
\beq
\label{e:SRN-nonlin}
\|\cN_S(v)\|_H \leq C \|v\|_H^r,
\eeq
where $r>1$ and $C$ may be chosen independent of $\vp\in\cP.$    
We assume that there exists a fixed value of $\delta>0$,  for which the quasi-steady manifold and the associated linearization 
satisfy the following hypotheses: 
\vskip 0.1in
\noindent(\textbf{H0}) The manifold $\cM$ is quasi-steady: that is, there exists $C_0>0$ such that for all 
$\vq\in\cP$,
\begin{equation}
\label{e:H0}
\norm{\cR(\vq)}_{H}\leq C_0\delta.
\end{equation}
\vskip 0.1in
\noindent
(\textbf{H1}) There exists $k_0, k_s>0$ such that for each $\vq\in\cP$ the spectrum of the operator  $\mbL_\vq$, viewed as a map from
$H$ into $X$ consists of a stable part $\sigma_s\subset\{\lambda\lvert{\lambda}\leq{-k_s}\}$  and a slow part 
$\sigma_0\subset\{\lambda\lvert{|\lambda|}\leq{c_0\delta}\}$. The associated slow eigenspace $Y_\vq$ has dimension $n$, equal to both
the dimension $\cP$ and to the tangent space to $\cM$.
\vskip 0.1in
\noindent
(\textbf{H2}) There exists $C_2>0$ such that for each fixed  $\vq\in\cP$, the operator $\mbL_\vq$ generates a $C_0$ semigroup 
$S_\vq$ which satisfies
\begin{equation}
\label{e:H2}
\norm{S_\vq (t)u}_{H}\leq C_2 e^{-k_st}\|u\|_{H}, 
\end{equation}
for all $t\geq0$ and all $u\in{Y_\vq'}\coloneqq Y_\vq^\perp \cap H$, where the perp is taken in the $X$ norm.

\vskip 0.1in
\noindent
(\textbf{H3}) For each $\vq\in\cP$, $Y_\vq$ is well-approximated by the tangent plane $\cT(\vq)$ of $\cM$ at $\vq$.  
Specifically, there exists a constant $C_3>0$ and an ordering $\{\psi_1, \ldots, \psi_n\}$ of the eigenfunctions of $Y_\vq$
such that 
\beq
\label{e:H3}
\left \| \psi_i(\vq) - \frac{\partial \Phi(\cdot;\vq)}{\partial p_i} \right\|_{H} \leq C_3 \delta, \hspace{0.50in} {\rm for}, i=1, \ldots, n,
 \eeq
 holds for all $\vq\in\cP.$
\vskip 0.1in
\noindent
(\textbf{H4})  There exists a constant $C_4>0$ such that  the normalized eigenvectors $\{\psi_1,\ldots,\psi_{n}\}$ of the $Y_\vq$ satisfy
\begin{equation}
\label{e:H4}
\max_{\substack{i=1,...,n\\ \vq \in\cP}}\left(\norm{\psi_i(\vq)}_{H}+\norm{\nabla_\vq^2\psi_i(\vq)}_{H}\right)\leq{C_4}.
\end{equation}
Under these hypotheses we have the following reduction.

\begin{thm}\cite[Theorem 2.1]{Prom-02}
\label{t:KP}
Suppose that the system (\ref{e:abstract_ev}) has a manifold $\cM$ for which the hypotheses (H0)-(H4) and (\ref{e:SRN-nonlin}) 
are satisfied for some $r>1$ and some $\delta>0$ sufficiently small. Then there exists $\eta_0$ and $M_0>0$, such that 
the solutions $u$ of \eqref{e:abstract_ev} corresponding to initial data $u_0$ that lie within an $\eta_0$-neighborhood of $\cM$ in $H$ 
can be decomposed as
\begin{equation}
\label{e:Phi-w-decomp}
u(t)=\Phi(\cdot,\vp(t))+w(\cdot,t),
\end{equation}
where the deviation $w\in Y_{\vp(t)}'$ satisfies
\begin{equation}
\label{e:w-eq}
\norm{w(\cdot,t)}_{H}\leq{M_0}({\eta_0}e^{-k_s(t-t_0)}+{\delta}) \quad for \quad t\in (0,\Texit).
\end{equation}
If $\vp(0)$ is an $O(1)$ distance to $\partial\cP$, then the exit time $\Texit\geq c_0 \delta^{-1}$.
After a transient time, $T_1=O(|\ln\delta/\eta_0|) \ll \Texit$, the deviation satisfies $\|w\|_{H}=\cO(\delta)$ and the
parameters $\vp(t)$ evolve at leading order via the closed system
\begin{equation}
\label{pulseODE}
\dot{p}_i=\Bigg\langle \cR(\vp),\frac{\partial \Phi}{\partial p_i}\Bigg\rangle_{X}+\cO(\delta^{1+r},\delta^2) \qquad for\quad t>T_1,
\end{equation}
for $i=1,\dots,n$. If the set $\cP$ is forward invariant under this flow, then we may take $\Texit=\infty.$
\end{thm}

\subsection{The Energy Landscape Framework}

We compare the scope and results of Theorem\,\ref{t:KP} with the energy landscape techniques introduced in \cite{OR-07} and 
refined in \cite{BFK-18}.  The GS approach uses the uniform coercivity of the energy in the directions normal to the 
quasi-steady manifold to develop an excluded zone which dynamically traps orbits in a thin neighborhood of the manifold.
Specifically,  the approach assumes an energy $\mrI:H\mapsto \R$, nested Hilbert spaces $H\subset X \subset H^*$, 
and an associated gradient system
\beq
\label{G-flow}
u_t = F(u):=-\cG \nabla_X\mrI (u),
\eeq
with the variational derivative of $\mrI$ taken in the $X$ norm.  It is often the case that the energy is naturally
formulated in the inner product on one space, $X$, while the gradient is calculated in a different inner product.
To emphasize this we have introduced the gradient operator $\cG$, a non-negative $X$-self-adjoint, linear operator that may possess a finite dimensional kernel. 
We assume that $\cG$ has an inverse that is uniformly bounded as a map, $\cG^{-1}:X_\cG \mapsto X_\cG$, where $\Pi_\cG$ is the $X$-orthogonal projection 
onto $X_\cG:=\textrm{ker}(\cG)^\perp.$ 
We introduce $\cG_1:=\cG^{\frac{1}{2}},$ and the associated inner product
\beq
\label{e:Gip-def}
\langle u, v\rangle_\cG:= \langle \cG_1^{-1} u, \cG_1^{-1}v\rangle_{X}.
\eeq
It straightforward to see that for $u\in H$ the variational derivative of $\mrI$ in the $\cG$-inner product satisfies the relation
$\nabla_\cG \mrI = \cG \nabla_X\mrI$, 
and hence (\ref{G-flow}) is the gradient flow of $\mrI$ in the $\cG$ norm. This flow decreases the energy,
\beq
\frac{d}{dt}\mrI(u(t)) = \left\langle \nabla_X \mrI , u_t\right\rangle_{X} =- \left\|\cG_1 \nabla_X\mrI \right\|_{X}^2
=- \left\|\nabla_\cG\mrI \right\|_{\cG}^2 \leq 0,
\eeq
and for any initial data $u_0\in H$ it leaves the space $u_0+X_\cG$ invariant. Indeed if $v\in\textrm{ker}(\cG)$ then
\beq
\frac{d}{dt}\langle u(t), v\rangle_{X} =- \left\langle \cG  \nabla_X \mrI (u), v \right\rangle_{X} = - \left\langle  \nabla_X \mrI(u),\cG v \right\rangle_{X}=0.
\eeq
The main result of the EL approach states that if $u\in H$ is sufficiently close to the quasi-steady manifold $\cM$, the manifold is normally $H$-coercive, 
and the energy of $u$ is low, then the $H$-distance of $u$ to $\cM$, denoted $d_H(u,\cM)$, is controlled by the energy, which is non-increasing, and hence
$u$ must remain close to manifold so long is it does not reach its boundary. 
In addition to the normal coercivity assumption, a key role is played by a projection onto the manifold. 

For simplicity of presentation we consider a less general framework than that presented in \cite{BFK-18}. Some of these
modifications arise from the fact that we have explicitly factored the variational derivative of $\mrI$ into a variational derivative 
in the base space $X$ and a linear gradient $\cG$. While this sacrifices some generality, it makes the relative independence of
the results upon the choice of gradient $\cG$ more explicit.

\vskip 0.1in
\noindent (\textbf{A0}) 
There exists  a smooth manifold $\cM$  embedded into the Hilbert space $H$, a $\delta_0>0$, and an energy $\mrI$ defined in $H$ on which the
energy has small variation, 
\beq | \mrI(u_1)-\mrI(u_2)| \leq \delta_0, \hspace {0.5in} \textrm{for}\,\textrm{all}\quad u_1, u_2\in \cM.
\label{e:A0}
\eeq

\vskip -0.01in
\noindent (\textbf{A1})  There exists a projection $\Pi_\cM$ on $\cM$, with complement $\tPi_{\cM}:=I-\Pi_{\cM}$, 
defined within an $H$-neighborhood of size $\eta>0$ of $\cM$ and a constant $c_1>0$ such that for all $u$ in the neighborhood
\beq
 \|\tPi_\cM u\|_\mrH \leq c_1 d_\mrH (\cM,u), 
\eeq
where $d_H$ denotes the $H$-norm distance function.

\vskip -0.02in
\noindent (\textbf{A2})
For all $u$ with $d_H(u,\cM)<\eta$, the functional $\mrI$ admits an $X$-variation expansion of the form
\beq
\label{e:A20}
 \mrI(u) = \mrI(\Pi_\cM u) + \left\langle \nabla_X\mrI(\Pi_\cM u), \tPi_\cM u\right\rangle_{X} + \left\langle \nabla^2_X\mrI (\Pi_{\cM}u)\tPi_\cM u, \tPi_\cM u\right\rangle_{X} + \cN_E(\tPi_\cM u),
 \eeq
 which satisfies the following: small residual,
 \beq 
 \label{e:A21}
\left  |\left\langle \nabla_X \mrI(\Pi_\cM u), \tPi_\cM u\right\rangle_{\!\! X}\right| \leq \delta_2 \|  \tPi_\cM u\|_\mrH,
 \eeq
$X$ to $H$ normal coercivity,
 \beq
 \label{e:A22}
 \left \langle \nabla^2_X\mrI(\Pi_{\cM}u)\tPi_\cM u , \tPi_\cM u\right\rangle_{\!\!X} \geq \mu_2 \|\tPi_\cM u\|_\mrH^2,
\eeq
and bounded nonlinearity,
 \beq
 \label{e:A23}
 | \cN_E(\tPi_\cM u)| \leq c_2 \| \tPi_\cM u\|_\mrH^\rho,
 \eeq  
for some $\delta_2,c_2>0$, some $\mu_2>0$, and $\rho>2.$

The result exploits the structure of the energy $\mrI$ and hence remarkably, is substantially independent of
the choice of the gradient $\cG$. The proof requires little more than the quadratic formula.  

\begin{thm} \cite[Theorem 2.1]{BFK-18}
\label{t:BFK}
Suppose there exists a choice of gradient $\cG$ for which
the energy $\mrI$, the manifold $\cM$, and the projection $\Pi_\cM$ satisfy (\textbf{A0})-(\textbf{A2}). Assume $u\in \mrH$ satisfies
\beq
\label{e:delta1-def}
  \mrI(u) \leq \sup\limits_{\Phi\in \cM} \mrI(\Phi) + \delta_1,
\eeq
for some $\delta_1>0$. Define
\beq 
\eta^* := \min \left\{\eta, \frac{1}{c_1}\left(\frac{\mu_1}{2c_2}\right)^{\frac{1}{s-2}}\right\}
\eeq
and 
\beq
\label{e:eta_*-def}
\eta_*:= \frac{\delta_2}{\mu_2} +\sqrt{\frac{\delta_2^2}{\mu_2^2} + 2 \frac{\delta_0+\delta_1}{\mu_2}}.
\eeq
If  $\delta_0$, $\delta_1$, and $\delta_2$ are small enough that $\eta_*<\eta^*$, then
\beq d_H(u,\cM)<\eta^* \implies d_H(u,\cM)<\eta_*.
\label{e:BFK}
\eeq
\end{thm}

The SRN and the EL techniques have non-trivial overlap in their applicability. We first consider the ``base-case'' in which the gradient $\cG$  
is taken to be the $X$-orthogonal projection onto a prescribed kernel.  We show that the SRN hypotheses imply the majority of the EL 
assumptions for this case, and develop two additional hypotheses, one for the SRN and one for the EL, under which the EL assumptions hold 
in their entirety. The first assumption simplifies the interaction of the manifold and the kernel of the gradient, and the second 
mirrors standard interpolation results used to boost coercivity into the strong norm.   The result, Theorem\,\ref{t:KP-BFK}, 
emphasizes that the EL approach holds for a large class of gradients which share the same kernel.  
The second main result, given in Section\,\ref{s:GISF} develops additional assumptions on the gradients for which the 
SRN may be extended beyond the base-case gradient. This extension requires a non-trivial reformulation of the problem to symmetrize the
gradient flow linearization $\mbL$.  

\vskip 0.1in
\noindent (\textbf{EH1})
Let $u_0$ denote the initial data to (\ref{G-flow}), the manifold $\cM$ lies in the invariant plane $u_0+X_\cG.$
\vskip 0.1in
\noindent (\textbf{EA})
There exist positive parameters $\mu_e, \gamma_e$ such that for all $\Phi\in\cM$ we have
\beq
\label{e:equiv-norm}
 \left \langle \left(\nabla^2_X\mrI(\Phi)+\gamma_e\right) v, v\right\rangle_X \geq \mu_e \|v\|_H^2,
 \eeq
 for all $v\in H\cap X_\cG$ and all $\Phi\in \cM.$
 \vskip 0.1in
 
 \begin{remark}
 The assumption (\textbf{EH1}) implies that $\cT_\vp\subset X_\cG$ for all $\vp\in\cP$. One way to satisfy this assumption is
 to insert extra parameters, $\tvp$ into the ansatz $\Phi=\Phi(\vp,\tvp)$, and constrain $\vp$ and $\tvp$ to enforce $\Pi_\cG (\Phi-u_0)=0.$
 The key is to show that the reduced family of parameters satisfies the remaining hypotheses. This approach is employed in
 Section\,\ref{s:FCH}.  
 \end{remark}
 
To establish a non-trivial overlap between the assumptions of the SRN and the EL approaches we show that 
(\textbf{H0})-(\textbf{H4}), together with (\textbf{EH1}) and (\textbf{EA}),  imply  (\textbf{A1})-(\textbf{A2}). 
While the assumption (\textbf{A0}) is not required for the SRN approach, we show that there is a wide class of gradients
for which the EL approach applies. Indeed we fix a finite co-dimension  space $X_0\subset X$ with orthogonal projection $\Pi_0:X\mapsto X_0$ and a quasi-steady manifold $\cM$ and consider the class $\cC_{X_0}$ of non-negative, $X$-self adjoint gradients 
\beq
\label{e:cC-def}
 \cC_{X_0}=\{ \cG:H\mapsto X_0\, \bigl |\,  \textrm{ker}(\cG)=X_0^\perp; \cG^{-1}:X_0\mapsto \cD_\cG\subset X_0, X\textrm{-norm}\,\textrm{bounded}\}.
\eeq 
We show that the choice of gradient from this class has limited impact on the slow-flow result associated to the underlying low-energy manifold.

\begin{thm}
\label{t:KP-BFK}
Fix the space $X_0$ and the class of gradient $\cC_{X_0}$ as in (\ref{e:cC-def}). Suppose that the energy $\mrI$ and the 
manifold $\cM$ correspond to the framework of (\ref{G-flow}). 
If the hypotheses (\textbf{H0})-(\textbf{H4}), (\textbf{EH1}), and (\textbf{EA}) hold for this system with the gradient $\cG=\Pi_0$, then 
there exists a projection $\Pi_\cM$ for which (\textbf{A1})-(\textbf{A2}) are valid. Moreover assume (\textbf{A0}) holds and initial data $u_0$ satisfies 
(\ref{e:delta1-def}) with $\delta$, $\delta_0$, and $\delta_1$ sufficiently small that $\eta_*<\eta^*$. Then the corresponding solution 
$u(t)$ of (\ref{G-flow}) can be decomposed as in (\ref{e:Phi-w-decomp}) where the residual $w$ satisfies $\|w\|_H\leq \eta_*$ 
for all $t\in(0,\Texit)$ where $\Texit:= \inf\{t\,\bigl|\, \vp(t)\notin\cP\}.$
\end{thm}

\begin{remark}
If the assumptions of Theorem\,\ref{t:KP} hold for the gradient $\cG=\Pi_{X_0}$ then one recovers the attraction of an $O(\eta_0)$ $H$-neighborhood
of $\cM$ into an $O(\delta)$ $H$-neighborhood of the manifold, as well as the leading order asymptotics of the flow projected onto the
tangent plane of the manifold, so long as $\vp\in\cP$. For the flows produced by the other gradients $\cG\in\cC_{X_0}$ one recovers the forward invariance of a generically wider $O(\eta_*)$  $H-$neighborhood of $\cM$, up to the boundary of $\cM$, however the decomposition of the solution
into modes tangential and normal to $\cM$ is generically not accurate enough to recover the leading order projection of $u_t$ onto the
tangent plane of the $\cM$, but do afford lower bounds on the exit time, as given in \cite[Theorem 2.2]{BFK-18}.
\end{remark}
\begin{proof}
We assume the existence of a quasi-steady manifold, $\cM$ that verifies (\textbf{H0})-(\textbf{H4}) for $F=\Pi_0 \nabla_X J$.  
The existence of the projection $\Pi_\cM$ is established in Proposition\, 2.2 of \cite{Prom-02}. In particular this result establishes
the existence of an $\eta_0>0$ for which  $u\in X$ with $d_X(u,\cM)\leq \eta_0$ can be decomposed as
$u=\Phi(\vp_*)+\eta_0 \hat{W}_0$, with $\|\hat{W_0}\|_X\leq 1$. Moreover it
establishes the existence of a function $\hat{\vp}=\hat{\vp}(u)=\vp_*+\eta_0 \cH(\hat{W})$ with $\cH(0)=0$, $\cH$
smooth in the $H$ norm, and for which the projection $\Pi_\cM u := \Phi(\hat{\vp}(u))$ enjoys the property 
$\tPi_\cM u \in \cT_\cM^\perp(\hat{\vp}).$  By the triangle inequality we deduce that
\beq
\|\tPi_\cM u\|_H \leq \| u-\Phi(\vp_*)\|_H + \| \Phi(\vp_*)-\Phi(\hat{\vp})\|_H= d_H(u,\cM) + \| \Phi(\vp_*)-\Phi(\hat{\vp})\|_H.
\eeq
Since $\cH$ is smooth there exists $M_0>0$ such that 
$$|\vp_*-\hat{\vp}|\leq \eta_0 M_0 \|\hat{W}_0\|_H\leq M_0 d_H(u,\cM).$$
Since $\Phi$ is a smooth function of $\vp$ we deduce that (\textbf{A1}) holds with $\eta=\eta_0$ for $\eta_0$ sufficiently small. 

For the gradient flow, (\ref{G-flow}), the choice of gradient $\cG=\Pi_0$ reduces to the identity on $X_0$. This affords the 
identification
\beq
\Pi_0 \nabla_X \mrI = -\cG^{-1} F(u)= -\Pi_0 F(u)=-F(u).
\eeq
As the space $u_0+X_0$ is invariant under the flow, it is sufficient to establish the bounds (\textbf{A2}) on $X_0$.
Indeed, writing $u=\Phi +v$ with $\Phi\in\cM$, by (\textbf{EH1}) we have $\Phi-u_0\in X_0$,  so that $\Pi_\cM u =\Phi\in X_0$ and 
$v=\tPi_\cM u\in X_0$. We may use the expansion (\ref{e:F-taylor}) to write
\beq
\Pi_0\nabla_X \mrI (\Phi+v) =-\Pi_{0}\cR(\vp)  - \Pi_{0}\mbL_\vp v -\Pi_{0}\cN_S(\Phi_\vp; v),
\eeq
where $\mbL$ denotes the linearization of the full gradients flow $F$ at $\Phi_\vp$.
 Comparing this with the expansion (\ref{e:A23}) and using the fundamental theorem of calculus we find for each $v\in H\cap X_0$, 
 that the expansion holds with
\beq
\cN_E(v) := -\int_0^1 \left\langle \cN_S(\Phi,sv),v\right\rangle_X\, \mathrm{d}s.
\eeq
Since the $H$-norm controls the $X$-norm, and $\cN_S$ satisfies (\ref{e:SRN-nonlin}) we  determine that (\ref{e:A23}) holds with 
$\rho=r+1>2$ on $X_\cG,$  which is consistent with the application of Theorem\,\ref{t:BFK}.  
Since $\nabla_X \mrI(\Phi(\vp))=-\Pi_0\cR(\vp)$, the bound (\ref{e:H0}) implies that the small residual assumption (\ref{e:A21}) holds with $\delta_2=c_0\delta$. To establish assumption (\textbf{A2}) it remains to verify the coercivity estimate (\ref{e:A22}) which we establish in Lemma\,\ref{l:coercivity}.

The second variation of $J$ at a point $\Phi_\vp$ on $\cM$ with perturbations taken from the constrained set $X _0$, induces the constrained operator
\beq
\label{e:cL-def}
\nabla_{X_0}^2 \mrI(\Phi(\vp))= -\Pi_0\mbL_\vp=-\Pi_0\mbL_\vp\Pi_0. 
\eeq
 
\begin{lem}
{\label{l:coercivity}} 
Assume  (\textbf{H0})-(\textbf{H4}), (\textbf{EH1}),  and (\textbf{EA}) hold then the manifold is normally $H$-coercive. That is exists a 
$\mu>0$ such that for all $\vp\in\cP$ the bilinear form (\ref{e:A22}) induced by the constrained second variation $\cL$ of $\mrI$ at 
$\Phi(\vp)$  satisfies
\begin{equation}
\left\langle{-\mbL v,v}\right\rangle_X\geq\mu{\|v\|^2_{H}},\label{H2-coer}
\end{equation}
for all $v\in \cT_{\vp}^{\perpX}. $
\end{lem}
\begin{proof}
By construction of the projection and (\textbf{EH1}), $\textrm{Range}(\tPi_\cM(\vp)) = \cT_{\vp}^{\perpX}\subset X_0.$ We first establish $X$ coercivity of 
$-\mbL$ on $\cT_{\vp}^{\perpX}$ by finding a $\tilde{\mu}>0$ such that 
\begin{equation}
\langle -(\mbL-\tilde{\mu})v,v\rangle_X \geq 0,
\end{equation}
for all $v\in \cT_{\vp}^{\perpX}$. We introduce the bilinear form
\beq
b[v,w]:= \langle -(\mbL-\tilde{\mu})v,w\rangle_X,
\eeq
associated to $-(\mbL-\tilde{\mu})$. Restricting the bilinear form to $\cT_{\vp}^{\perpX}$, induces the constrained operator 
$-\tPi_\cM (\mbL-\tilde{\mu}) \tPi_\cM.$  We remark from hypothesis (\textbf{H1}) that $-\mbL_\vp$ has a finite number of negative
eigenvalues.  The  $X$-coercivity of $-\mbL$ is equivalent to the the statement $\mn(-\tPi_\cM (\mbL-\tilde{\mu}) \tPi_\cM) =0$, where 
the negative index $\mn(L)$ denotes the number of negative eigenvalues of a self-adjoint operator $L$ counted according to multiplicity.

We apply Proposition 5.3.1 of \cite{KapProbook}, which equates  the number of the negative eigenvalues of a constrained operator to the difference of the number of the negative 
eigenvalues of the operator and an associated constraint matrix. More specifically, given an invertible, $X$-self-adjoint operator $L$ and an orthogonal projection $\Pi_V$ onto a finite-codimension 
subspace $V\subset X$.  Then the number of negative eigenvalues of the constrained operator $\Pi_V L \Pi_V$, as a map from $V\mapsto V$, is given by
\beq
\mn (\Pi_V L\Pi_V ) =\mn(L)-\mn(D),\label{constaint-index}
\eeq
where the finite-dimensional constraint matrix $D$ is defined by
\begin{equation}
D_{ij} \coloneqq \langle s_i, L^{-1} s_j \rangle, \hspace{0.5in} \textrm{for}\, i,j =1, \ldots, n
\label{constraint-matrix}
\end{equation}
where $\{s_i\}_{i=1}^n$ is a basis for $V^{\perpX}$. We apply this theorem with $L=-(\mbL-\tilde{\mu})$, $X=X_0$, and $V=\cT_{\vp}.$ 
From (\textbf{H1}), for $\tilde\mu\in(k_s/2, k_s)$, we have $\mn(-(\mbL-\tilde{\mu}))=n.$

To determine $\mn(D(\tilde{\mu}))$, from \textbf{(H1)} and \textbf{(H3)} the slow-space eigenfunctions of $-(\mbL-\tilde{\mu})$  
take the form $\psi_{i}= s_i +\psi_i^\perp$ where $\|\psi_i^\perp\|_H = \cO(\delta)$, and $s_i:= \frac {\partial\Phi}{\partial p_i}.$ 
We denote the slow-eigenvalues of $\mbL$ by $\{\lambda_1, \ldots, \lambda_n\}$. Since $-(\mbL-\tilde{\mu})$ has an $\cO(1)$ inverse we 
deduce that
\begin{equation}
D_{ij}(\tilde{\mu})=\left\langle s_i, -(\mbL-\tilde\mu)^{-1} s_j \right\rangle_X =\left \langle s_i, -(\lambda_i-\tilde{\mu})^{-1}\phi_i \right\rangle_X+\cO(\delta)= \frac{-1}{\lambda_i-\tilde{\mu}} \delta_{ij} + \cO(\delta).
\end{equation}
From  \textbf{(H1)} we have $|\lambda_i|=\cO(\delta)$ and hence $D(\tilde{\mu})=\frac{1}{\tilde{\mu}} I_{n\times n} +\cO(\delta)$ and $\mn(D(\tilde{\mu}))=n.$ From the variational
formulation of eigenvalues we deduce that 
\beq
 \left\langle{-\mbL v,v}\right\rangle_X\geq\tilde{\mu}\|v\|^2_X.
 \eeq
 for $v\in \cT_{\vp}^{\perpX}$.
 
To establish the $H$ coercivity. We introduce $\alpha\in(0,1)$ and write
 \beq
  \left\langle{-\mbL v,v}\right\rangle_X = \alpha \left(   \left\langle{-\mbL v,v}\right\rangle_X + \frac{1-\alpha}{\alpha}   \left\langle{-\mbL v,v}\right\rangle_X\right) \geq 
                               \alpha \left(   \left\langle{\cL v,v}\right\rangle_X + \frac{(1-\alpha)\tilde{\mu}}{\alpha}   \left\| v\right\|_X^2\right). 
 \eeq
 Choosing $\alpha= \frac{\tilde{\mu}}{\tilde{\mu}+\gamma_e}$ we have $ \frac{(1-\alpha)\tilde{\mu}}{\alpha}=\gamma_e$. Applying  
 (\ref{e:equiv-norm}) of (\textbf{EA}) we deduce
\beq 
 \left\langle{-\mbL v,v}\right\rangle_X \geq  \frac{\tilde{\mu}}{\tilde{\mu}+\gamma_e} \mu_e \|v\|_H^2,
 \eeq
  which establishes (\ref{H2-coer})  with $\mu =  \frac{\tilde{\mu}\mu_e }{\tilde{\mu}+\gamma_e}.$ 
\end{proof}

Returning to the proof of Theorem\,\ref{t:KP-BFK}, we consider (\ref{G-flow}) with any gradient $\cG\in\cC_{X_0}$ and
deduce that Theorem\,\ref{t:BFK} holds with $\eta^*=\eta_0$ as given by Theorem\,\ref{t:KP} and $\eta_*$ given by (\ref{e:eta_*-def}) 
so long as $\delta, \delta_0, $ and $\delta_1$ are sufficiently small that $\eta_*<\eta_0$. From Theorem\,\ref{t:BFK} it follows that
the solution $u=u(t)$ of (\ref{G-flow}) can be decomposed as $u(t)=\Phi(\vp(t))+ w$ where $w=\tPi_\cM u(t)$ satisfies $\|w\|_H\leq \eta_*$,
so long as $\vp\in\cP.$ 
 \end{proof}
 
\subsection{Gradient Invariance of Slow Flows}
\label{s:GISF}
We extend the applicability of the SRN approach to a class of gradients the includes $\Pi_0$, and shares its kernel. This
class is more restrictive than $\cC_{X_0}$ given in (\ref{e:cC-def}).  For for all $t>0$ the solution $u$ of (\ref{G-flow}) satisfies 
$u(t)-\cM\in X_0$.   This motivates the decomposition
\beq 
\label{e:cGi-decomp}
u = \Phi(\cdot;\vp) +  \rho^{-1} \cG_1 w,
\eeq
where $w\in H_{\cG_1}\subset X_0$ satisfies $w\bot \cG_1^{-1} \cT.$
The scaling parameter $\rho\gg 1$ is included to allow the incorporation of singularly perturbed energies such as the FCH whose 
differential operators are homogeneously scaled by the small parameter $\eps\ll1$.  The operator $\cG_1$ is defined as the
square root of $\cG$ and the space $H_{\cG_1}$ denotes the functions in $H$ for which the norm  $\|w\|_{H_{\cG_1}}:= \|\cG_1 w\|_H,$
is finite. 

With this decomposition we re-write the gradient flow
\beq
\label{e:G-flow2}
 u_t = -\cG_1^2 \nabla_X \mrI(u),
\eeq
as
\beq
\label{e:cG-w}
 \rho \cG_1^{-1} \nabla_\vp \Phi \cdot \dot{\vp} + w_t = -\rho \cG_1\cR - \cG_1 \cL \cG_1 w - \rho \cG_1\cN_S(\rho^{-1}\cG_1 w),
 \eeq
 where $\cL=\nabla_XJ(\Phi_\vp)$ is the second variation of $J$ in the $X$-inner product.
The key point is that the linear operator $\mbL:=\cG_1\cL\cG_1$ has been symmetrized and the nonlinearity has been scaled.
Indeed, comparing to the base case $\cG=\Pi_0$, we see that the tangent plane $\nabla_\vp\Phi$ has been scaled and mapped to $\cG_1^{-1}\nabla_\vp\Phi$,
and the residual is scaled and mapped by $\cG_1$.
 
We have the following immediate result
\begin{cor}
\label{c:cG-coercivity}
There exists $\mu_\cG>0$ such that the bilinear form
$$b_{\cG_1}(w,w):= \langle \cG_1\cL\cG_1 w,w \rangle_X \geq \mu  \|\cG_1 w\|_H^2\geq \mu_\cG\|w\|_H^2,$$
for all $w\in(\cG_1^{-1} \cT)^\bot\cap H_{\cG_1}$. Here $\mu$ is the coercivity constant from Lemma\,\ref{l:coercivity}.
\end{cor}
\begin{proof}
Since $w\in (\cG_1^{-1} \cT)^\bot\cap X_0$, we have $w=\cG_1^{-1} v$ where $v\bot \cT.$ In particular
$$  \langle \cG_1\cL\cG_1 w,w \rangle_X = \langle \cL v, v\rangle_X \geq \mu \|v\|_H^2= \mu \|\cG_1 w\|_H^2\geq \frac{\mu}{M^2} \|w\|_H^2,$$
where $M$ is the bound on $\cG_1^{-1}:X_0\cap H\mapsto X_0\cap H$.
\end{proof}
Without loss of generality we may rescale  both $\cG$ and the temporal variable so that the $X$-operator norm of $\cG_1^{-1}$ is 
bounded sharply by the constant $1$ on its domain $X_0$. To recover the leading order reduced flow we require two extra assumptions
that constrain the choice of $\rho$, which must satisfy $\delta_\cG:=\delta \rho^3 \ll1.$
\vskip 0.1in
\noindent (\textbf{EH2})
There exists $c>0$, independent of $\rho\gg 1$ for which the nonlinearity $\cN_S$ introduced in (\ref{e:SRN-nonlin}) satisfies
\beq
\label{e:EH2}
\rho \|\cG_1\cN_S(\rho^{-1} \cG_1 w) \|_{H_{\cG_1}} \leq c \|w\|_{H_{\cG_1}}^2.
\eeq
\vskip 0.1in

\noindent (\textbf{EH3})
There exists a constant $c>0$, independent of $\rho$, for which the following estimates
\beq
\label{e:EH3-1}
    \|\cG_1 \nabla_X \mrI(\Phi(\vp))\|_{H_{\cG_1}} \leq c \rho^2 \delta,
\eeq
and
\beq
\label{e:EH3-2}
 \hspace{1.5in}  \| \cG_1 u\|_X \leq c \rho \|u\|_X,\hspace{1.0in} \forall  u \in\cT_\vp,
\eeq
hold uniformly for $\vp\in\cP.$
\vskip 0.1in

\begin{thm}
\label{t:dyn-equiv}
Assume that Theorem\,\ref{t:KP} and its hypotheses hold for the choice of gradient $\cG=\Pi_0.$ If in addition
hypotheses (\textbf{EH2}) and (\textbf{EH3}) hold for parameters $\rho$ and $\delta$ satisfying $\rho\gg1$ and 
$\delta_\cG:=\delta \rho^3\ll1$, then the flow (\ref{e:cG-w}) satisfies the hypotheses (\textbf{H0})-(\textbf{H4}) for the pair $H_{\cG_1}\subset X$ with 
$\delta$ replaced by $\delta_\cG$ and a reparameterization of the manifold $\cM$ through a smooth transformation $\tvp=\tvp(\vp)$. The solution 
$u$ of (\ref{G-flow}) can be decomposed as (\ref{e:cGi-decomp}) where $\tw:=\rho^{-1} w$ satisfies the bounds (\ref{e:w-eq}) in the norm $H_\cG$ and the rescaled parameters $\tvp$ satisfy
\begin{equation}
\label{PulseODE-cG1}
\dot{\tp}_i=\Bigg\langle \cR(\tvp),\frac{\partial \Phi}{\partial \tp_i}\Bigg\rangle_{X}+\cO(\delta_\cG^{1+r},\delta_\cG^2).
\end{equation}
\end{thm}

\begin{remark}
Within this framework the impact of the change of gradient in to rescale the pulse dynamics. As we demonstrate explicitly in Section\,\ref{s:GISF}, 
for simple manifolds this rescaling can be uniform across the manifold, in  which case it amounts to a linear scaling of time.
\end{remark}
\begin{proof}
Since $\cG_1^{-1}\cT$ is an $n$ dimensional space, Corollary\,\ref{c:cG-coercivity} that $n(\cG_1\cL\cG_1-\mu_\cG)\leq n$.  The main step
to establish the hypotheses (\textbf{H0})-(\textbf{H4}) for the general gradient flow is to show that the operator $\cG_1\cL\cG_1$ retains
its spectral gap. To this end consider the eigenvalue problem
$$ \cG_1\cL\cG_1 \Psi = \lambda \Psi.$$
For $\lambda\in\sigma(\cG_1\cL\cG_1)\cap [-\infty,\mu_\cG)$ we decompose the eigenfunction as
\beq
\label{e:Psi-decomp}
 \Psi= \cG_1^{-1}\phi+ \Psi^\bot,
\eeq
where $\phi$ lies in $Y_\vp$ and $\Psi^\bot\bot \cG_1^{-1}Y_\vp.$ Projecting the eigenvalue problem onto $\cG_1^{-1}\phi$ we have
\beq
\label{e:PGIphi}
\langle \cL \phi,\phi\rangle + \langle \cG_1 \Psi^\bot, \cL\phi\rangle_X = \lambda \|\cG_1^{-1}\phi\|_X^2.
\eeq
Isolating $\lambda$ and bounding the first inner product with(\textbf{H1}), we use (\textbf{EH3}) and Raleigh-Ritz to obtain
\beq
\label{e:PGIphi2}
|\lambda| \leq \frac{c_0\delta\|\phi\|_X^2 + c_0 \delta \|\cG_1\phi\|_X\|\Psi^\bot\|_X}{\|\cG_1^{-1}\phi\|_X^2}\leq
c_0\delta\rho^2\left(1+ \frac{\|\Psi^\bot\|_X}{\|\cG_1^{-1}\phi\|_X}\right).
\eeq
Projecting the eigenvalue relation onto $\Psi^\bot$ yields
$$\langle \cL\phi, \cG_1\Psi^\bot\rangle_X + \langle\cG_1 \cL\cG_1 \Psi^\bot, \Psi^\bot\rangle_X = \lambda \|\Psi^\bot\|_X^2.$$
Using the coercivity result on the second term and  applying (\textbf{H1}) and (\textbf{EH3}) to the first term on the right-hand side
we find that
$$ (\mu-\lambda)\|\Psi^\bot\|_{H_{\cG_1}}\leq c_0\delta\rho^2 \|\cG_1^{-1}\phi\|_X.$$
In particular we bound
$$\frac{\|\Psi^\bot\|_X}{\|\cG_1^{-1}\phi\|_{X}} \leq \frac{\|\Psi^\bot\|_{H_{\cG_1}}}{\|\cG_1^{-1}\phi\|_{X}} \leq 
\frac {c_0 \delta\rho^2}{(\mu_\cG-\lambda)}.$$
With the normalization $1=\|\Psi\|_X^2= \|\cG_1^{-1}\phi\|_X^2+ \|\Psi^\bot\|_X^2,$ the estimate above and (\ref{e:PGIphi2}) imply that
\beq
\label{e:PGIphi3}
|\lambda|  +\frac{\|\Psi^\bot\|_{H_{\cG_1}}}{\|\cG_1^{-1}\phi\|_X}\leq  c\delta\rho^2\leq c \delta_\cG,
\eeq
This shows that $\lambda\in \sigma(\cG_1\cL\cG_1)$ and $\lambda <\mu_\cG$ implies
that $|\lambda|< c \delta_\cG \ll \mu_\cG$, which establishes the spectral gap. Moreover, to leading order in $\delta_\cG$, the operator 
$\cG_1^{-1}$ maps the slow  eigenfunctions of $\cL$ onto the slow eigenfunctions of $\cG_1\cL\cG_1$, even though this relation does 
not generically hold for the eigenfunctions of the stable spectrum. 

We assume that the hypotheses (\textbf{H0})-(\textbf{H4}) and (\ref{e:SRN-nonlin}) hold for the system with gradient $\Pi_0$ and
verify that they hold for the flow (\ref{e:G-flow2}), written in the form (\ref{e:cG-w}). This amounts to the replacement of the spaces $X=X$, $H=H_{\cG_1}$, the small parameter $\delta$ with $\delta_\cG$, the residual $\cR$ with $\cR_{\cG_1}:= \rho\cG_1 \cR$ and
the role of the tangent plane $\cT$ with $\cG_1^{-1}\cT.$  The equivalent of estimate (\ref{e:SRN-nonlin})  for the nonlinear term of 
(\ref{e:cG-w}) follows  immediately by assumption (\textbf{EH2}).
The hypotheses (\textbf{H0}) with bound $\delta_\cG$ holds for $\cR_{\cG_1}$ from assumption (\ref{e:EH3-1}) of (\textbf{EH3}). 
Since the eigenfunctions $\{\psi_i\}_{i=1}^n$ of $\cL$ are orthonormal in $X$, we deduce that the dim$(\cG_1^{-1}Y_\vp)= n$.
Motivated by (\ref{e:PGIphi3}) we may introduce the slow space $Y_{\vq,\cG_1}$ associated to $\cG_1\cL\cG_1$ with $k_s=\mu_\cG$. 
Since the bilinear form $b_{\cG_1}$ introduced in Corollary\,\ref{c:cG-coercivity} satisfies $b_{\cG_1}(u,v)\leq c_0\delta$ for all 
$u,v\in\cG_1^{-1}Y_\vp$ we deduce that dim$Y_{\vp,\cG_1}=n$ and that  (\textbf{H1}) holds.  
The operator $\cG_1\cL\cG_1$ constrained to act on $\cG_1^{-1}\cT_\vp\cap X_0$ is self-adjoint and has its spectrum contained in $(k_s,\infty)$.
It follows that the resolvent of $-\cG_1\cL\cG_1$ is uniformly bounded on the set $\{\Re \lambda < k_s\}$ and hence the semigroup 
$S_\vq$ associated to $-\cG_1\cL\cG_1$ is analytic and satisfies (\ref{e:H2}).
The slow eigenfunctions $\{\Psi_i\}_{i=1}^n$ of $\cG_1\cL\cG_1$ satisfy 
$$\Psi_i= \cG_1^{-1} \frac{\partial \Phi}{\partial p_i} + \Psi_i^\bot,$$ 
where $\Phi$ is smooth and the error term $\Psi_i^\bot$ satisfies the bound (\ref{e:PGIphi3}). Since 
$$\frac{\partial \cG_1^{-1} \Phi}{\partial p_i} = \cG_1^{-1} \frac{\partial \Phi}{\partial p_i},$$ 
and since (\textbf{H3}) holds with gradient $\Pi_0$, the bounds (\ref{e:PGIphi3}) establish (\textbf{H3}) for $\cG_1$, that is up to
a reparameterization of $\vp$, the bound (\ref{e:H3}) holds with $\partial_{p_i}\Phi$ replaced with $\partial_{\tp_i}\cG_1^{-1}\Phi$ and 
with $\delta$ replaced with small parameter $\delta_\cG$. 
Since the operator $\cG_1^{-1}$ is uniformly bounded on $H$,  the reparameterization  $\tvp$ of $\cM$ is uniformly smooth in $\vp$ the Hessian $\cG_1^{-1}\nabla^2_{\tvp} \Phi$ is bounded in the $H_{\cG_1}$ norm. The assumption (\textbf{H4}) follows.

The ODE  (\ref{PulseODE-cG1}) arises from the projection of the linear terms in(\ref{e:cG-w}) onto the small eigenspace, $Y_\vp$, of 
$\cG_1\cL\cG_1$.  The factors of $\rho$ cancel out, and the action of $\cG_1$ on $\cR$ is canceled by the $\cG_1^{-1}$ prefactor that 
maps $Y_\vp$ for $\Pi_0$ onto the leading order form of $Y_\vp$ for $\cG_1.$ The error terms arise from the bound on $\|w\|_{H_{\cG_1}}$
which follows from the estimates on the decomposition analogous to (\ref{e:w-eq}).
\end{proof}

\section{Pulse Dynamics and Gradient Invariance in FCH Gradient Flows}
\label{s:FCH}
We apply the results of Section 2 to gradient flows of FCH energy (\ref{e:FCH}) on the bounded domain $[0,\dell]\subset\R$.
For simplicity of presentation we set $\eta_1=\eta_2=0$, as these parameters have limited impact in the one-dimensional setting.

\subsection{Construction of the $n$-pulse Quasi-steady Manifold}
Introducing  the inner scaling $z=\frac{x}{\eps}$, we re-write the FCH as
\begin{equation}
\begin{aligned}
\mrI\left(u\right)=\displaystyle\int_{0}^{\frac{\dell}{\varepsilon}}\frac{1}{2}\left(\partial_z^2{u}-W'(u)\right)^2dz,  \label{sFCH0}
\end{aligned}
\end{equation}
and subject it to the mass constraint
\begin{equation}
\displaystyle\int_{0}^{\frac{{\dell}}{\eps}}(u-b_-)dz=M. 
\label{massconst}
\end{equation}
It is natural to consider $\mrI$ acting on admissible functions that satisfy the mass constraint and first-order Neumann boundary conditions 
\begin{equation}
\mathcal{A}=\bigg\{u\in H^2\left(\left[0,{\frac{\dell}{\eps}}\right]\right) \bigg\lvert \displaystyle\int_{0}^{{\frac{\dell}{\varepsilon}}}(u-b_-)dz=M, u_z(0)=u_z\left({\frac{\dell}{\eps}}\right)=0\bigg\}.
\end{equation}
The critical points of the inner scaling of FCH over the admissible space $\mathcal{A}\cap{H^4}\left(\left[0,{\frac{\dell}{\varepsilon}}\right]\right)$  are the solutions to the Euler-Lagrange equation 
\begin{equation}
\begin{cases}
&\qquad\nabla_X \mrI \coloneqq\left(\partial_z^2-W''(u)\right)\left(\partial_z^2{u}-W'(u)\right)=\lambda_\varepsilon, \\
&\partial_z^3{u}(0)=0,\partial_z^3{u\left({\frac{\dell}{\varepsilon}}\right)}=0, \partial_z{u}(0)=0, \partial_z{u}\left({\frac{\dell}{\varepsilon}}\right)=0,
\end{cases}\label{ELeq}
\end{equation}
where $\nabla_X$ is the first variational derivative of $\mrI$ with respect to $L^2$ inner product and $\lambda_\varepsilon$ is the $\eps$-dependent Lagrange multiplier. The no-flux boundary conditions arise naturally from the Euler-Lagrange formulation. To leading order the low-energy manifold is constructed from solutions
\begin{equation}
\label{e:freeway}
\partial_z^2{u}-W'(u)=0,
\end{equation}
that satisfy the no-flux boundary conditions. Classical phase-plane arguments show that \eqref{e:freeway} supports a homoclinic solution 
satisfying $\phi_h\rightarrow{b_-}$ as $z\rightarrow{\pm\infty}$. The n-pulse Ansatz, defined on all of $\R$, is given by
\begin{equation}
{u}_n\coloneqq b_-+ \sum_{j=1}^n \overline{\phi}_h\left(z-p_j\right), \label{78}
\end{equation}
where $\overline{\phi}_h\coloneqq\phi_h-b_-$ and ${\mathbf{p}}=(p_1,p_2,...,p_n)^{t}\in\R^n$ is the vector of pulse locations. The admissible set of pulse locations is given by
\begin{equation}
\cP:=\{{\mathbf{p}}\in\R^n:p_i<p_{i+1}\quad \textit{for} \quad i=0,...,n \quad \textit{and} \quad \Delta\mathbf{p}\geq \ell\},\label{pset}
\end{equation}
where $\Delta{\mathbf{p}}:=\min\limits_{i\neq{j}}|p_i-p_j|,$ and the boundary pulse locations $p_0$ and $p_{n+1}$ are introduced below. 
The pulse spacing parameter $\ell>0$ will be chosen sufficiently large that the exponential tail-tail interaction
terms $\delta:=e^{-\sqrt{\alpha_-}{\ell}}$ arising in the calculations are small compared to $\eps$. In particular this implies that $\ell\gg |\ln\eps|$. 

To complete the definition of the pulse manifold we introduce the operator
\begin{equation}
\mrL\coloneqq\partial_z^2-W''(\phi_h),
\end{equation}
corresponding to the linearization of \eqref{e:freeway} about $\phi_h$, as well as the operator
\begin{equation}
\mrL_n(\vp)\coloneqq\partial_z^2-W''\left({u}_n\right), \label{e:Ln-def}
\end{equation}
with both acting on the \emph{unbounded} domain $H^2(\R)$. To accommodate the mass constraint into the pulse ansatz we introduce 
$\mB_{j}\in{L^\infty(\R)}$ for $j=1,2$ as the solutions of
\begin{equation}
\mrL^j\mB_j=1,
\end{equation}
that are orthogonal to the kernel of $\cL$. These functions can be decomposed as
\begin{equation}
\mB_j=\omB_j+\mB_{j,\infty},
\end{equation}
where $\omB_j\in L^2(\R)$ decays exponentially to zero and the constant $\mB_{j,\infty} = (-\alpha_-)^{-j}$ where  $\alpha_-=W''(b_-)>0$. 
We introduce the background correction
\begin{equation}
\mB_{j,n}(z;\vp)\coloneqq \mB_{j,\infty}+ \sum\limits_{i=1}^n \omB_j(z-p_i),
\end{equation} 
and the boundary correction
\beq
\label{e:E-def}
\mrE(z;\vp):=  (1+e_0z)e^{-\sqrt{\alpha_-}(z-p_0)} + (1+e_{n+1}z)e^{\sqrt{\alpha_-}(z-p_{n+1})}.
\eeq
The full n-pulse ansatz takes the form
\begin{equation}
\label{Phi-def}
\Phi(z;\vp)\coloneqq u_n(z;\vp)+ \delta \lambda\mB_{2,n}(z;\vp)+\mrE(z,\vp,\lambda).
\end{equation}
The parameters in the boundary correction $\mrE$ are chosen dynamically to satisfy the four boundary conditions in  (\ref{ELeq}) while
the Euler-Lagrange parameter $\lambda$ is chosen dynamically to enforce the prescribed total mass constraint,
\beq
\label{Mass-eq}
\int_0^{\dell/\eps} \Phi(z;\vp)dz = M.
\eeq
Based upon Lemma\,\ref{l:App}, in the Appendix we can write this in the form
\begin{equation}
\Phi(z;\vp)= u_n(z;\vp)+ \delta P,
\label{Phi-def_s}
\end{equation}
where the perturbations $P$ are uniformly bounded in $H^4(0,\dell/\eps).$ Through these relations, the five internal parameters
$\tvp:= (p_0, p_{n+1}, e_0, e_{n+1}, \lambda),$ are prescribed as functions of $n$ pulse positions $\vp.$ To leading order, the boundary
pulse locations $p_0$ and $p_{n+1}$ are the reflection of $p_1$ and $p_n$ about the boundary points $0$ and $\dell/\eps$, respectively.  
The parameters $p_0$ and $e_0$ characterize the linearization of the two dimensional stable manifold of the fourth order system
$$ (\partial_z^2 - W''(u))(\partial_z^2 u - W'(u))=0,$$
at the equilibria $(b_-,0,0,0)$ while $p_{n+1}$ and $e_{n+1}$ characterize the linearization of the unstable manifold associated of this
system at $(b_-,0,0,0)$. 

The manifold of $n$-pulse solutions  with mass $M$ takes the form
\beq
\cM_{n,M}:=\{\Phi(\vp)\lvert \vp \in\mathcal{P} \,  \}.
\label{Mn-def}
\eeq 
The tangent plane to $\cM_{n,M}$ at $\Phi(\vp)$  takes the form 
\beq
\cT(\vp) =\mathrm{span}\left\{\frac{\partial \Phi(\vp)}{\partial{p_i}}\, \bigl|\, i=1,\ldots, n, \,\mathbf{p}\in\cP \right\}.
\eeq
 From Lemma\,\ref{l:App}  we have $\|E\|_{H^4}=O(\delta)$, and we calculate that
\beq
\cT(\vp) =\textrm{span}\Bigl( \left \{\phi_h'(z-p_i)+\delta \lambda \mB_2'(z-p_i)+\sqrt{\alpha_-}\delta_{1i} E_0-\sqrt{\alpha_-}\delta_{ni}E_{n+1}\right\}_{i=1}^n+O(\eps\delta,\delta^2)\Bigr)
\eeq
where $\delta_{ij}$ denotes the usual Kronecker delta function.

\subsection{Modulational Stability of $n$-pulses via SRN}
\label{chap:locmin}

We apply the SRN theorem  to the zero-mass gradient flow of FCH energy subject to
no-flux boundary conditions, obtaining the asymptotic attractivity and modulational stabilty of the n-pulse manifold.  
Specifically we set $X=L^2(0,\dell/\eps)$ and $H=H^4(0,\dell/\eps)$ subject to zero flux boundary conditions. 
We consider the $L^2$ mass-preserving gradient flow of the FCH,
\begin{equation}
\begin{aligned}
\label{e:FCH-flow}
&u_t=F(u):=-\Pi_0 \nabla_X \mrI(u),\\
&u(z,0)=u_0(z),
\end{aligned}
\end{equation}
where the zero-mass projection, $\Pi_0$, is defined as
$\Pi_0{f}\coloneqq{f}-\langle{f}\rangle_\dell$ with $\langle{f}\rangle_{\dell}$ denoting the average value of $f$ over $[0,\frac{\dell}{\eps}].$
This corresponds to the choice of gradient $\cG=\Pi_0$ and $X_\cG=\{1\}^\perp.$ 
The zero-mass projection gradient flow of the Cahn-Hilliard free energy modeling a phase separation process in a binary mixture was analyzed in \cite{RubStern}.

We consider solutions of (\ref{e:FCH-flow}) corresponding to initial data of the form 
\beq
\label{e:u0-init}
u_0= \Phi(z;\vp_0)+w_0(z).
\eeq
where $\vp_0\in\cP$ and $w_0\in H$ with $\|w_0\|_{H^4}$ sufficiently small, has zero mass, 
so that $u_0$ satisfies the boundary conditions and has mass $M$. 
We show that such initial data remain near $\mathcal{M}_M$ so long as they avoid its boundary, and during this time the solution satisfies a decomposition
\beq
u(t) = \Phi(\cdot ;\vp(t)) +w(t),
\eeq
and project the dynamics of \eqref{e:FCH-flow} onto the tangent plain of $\cM_{n,M}$ to derive an evolution for the pulse positions $\vp$ for which the remainder $w$, remains small. 
Moreover we identify small regions in the interior of $\cP$ associated to nearly equispaced pulse positions which the reduced flow (\ref{pulseODE}) leaves forward invariant. For initial data in these sets the exit time $\Texit=+\infty.$ 

We Taylor expand the the variational derivative of $\mrI$ about $\Phi(\vp)$  
\beq
\frac{\delta{\mrI}}{\delta{u}}(u)= \nabla_X \mrI(\Phi(\vp))+ \nabla_X^2 \mrI (\Phi) w+\cN_S(w).
\eeq
Using the expansion (\ref{Phi-def_s}) we identify leading order terms in the residual,
\beq
\begin{aligned}
\cR:= -\Pi_0 \frac{\delta{\mrI}}{\delta{u}}(\Phi(\vp)) &= -\Pi_0(\partial_z^2 -W''(u_n+\delta\, P))\left(\partial_z^2 u_n -W'(u_n) +\delta \mrL_n  P + O(\delta^2)\right),\\
 &= -\Pi_0\left(\mrL_n\cR_n + \delta \lambda \right)+ O(\delta^2),
 \label{e:resid}
 \end{aligned}
\eeq 
where we have introduced the $n$-pulse residual
\beq
\label{e:resid_n}
\cR_n(\vp):= \partial_z^2 u_n -W'(u_n).
\eeq
We denote the second variation of $\mrI$ as
\begin{equation}
\cL_\vp\coloneqq\nabla_X^2 \mrI=\left(\partial_z^2-W''(\Phi)\right)^2-\left(\partial_z^2 \Phi-W'(\Phi)\right)W'''(\Phi).\label{secvarI}
\end{equation}
We drop the $\vp$ subscript were doing so causes no ambiguity. Using the form of (\ref{Phi-def_s}) we expand \eqref{secvarI} about $u_n$ up to $\cO(\delta^2)$ terms
\begin{equation}
\begin{aligned} 
\cL&= \Bigl(\mrL_n- \delta W'''(u_n)P+\cO(\delta^2)\Bigr)^2- \Bigl(\cR_n+\delta\mrL_n P + \cO(\delta^2)\Bigr)\times\\
          &\quad \Bigl(W'''(u_n)+\delta W^{(4)}(u_n)P+\cO(\delta^2)\Bigr).
\end{aligned}
\end{equation}
From the Appendix we see that  $\mrL_n P= \lambda \mB_{n,1} +O(\delta)$, and expanding out the operators we find that
\beq
\cL=\mrL_n^2  -\delta\left( \mrL_n (W'''(u_n)P\cdot) -W'''(u_n)P \mrL_n\right) -W'''(u_n)(\cR_n +\delta \lambda \mB_1) + \cO(\delta^2).
\label{secvarI2}
\end{equation}
In particular the dominant term in $\cL$ is the positive semi-definite operator $\mrL_n^2$ with the lower order terms relatively compact with
respect to $\mrL^2_n.$  The bilinear form
\beq
b(u,v):= \langle \cL u, v\rangle_{L^2},
\eeq 
with $u,v\in H$, generated by the constrained operator $\Pi_0\cL\Pi_0$ which is self-adjoint. 
Indeed, the linearization $\mbL$ of the vector field $F=-\Pi_0 \nabla_X  \mrI$ at $\Phi$ takes the form 
\beq
\label{e:mbL-def}
\mbL=- \Pi_0 \cL.
\eeq 
Since the first projection in $\Pi_0\cL\Pi_0$ is superfluous when acting on $H$, $\mbL$ can be viewed as the negative of the 
generator of the bilinear form $b$ over $H$.  Consequently the spectrum of both $\mbL$ and $\cL$ are real and the adjoint eigenfunctions agree with 
the eigenfunctions, with the exception of the kernel of $\mbL$ given at leading order by $\mB_2$ while the kernel of $\mbL^\dag$ is spanned by $1$. 
We scale the eigenfunctions of $\cL$ to have $X$ norm one.  

\subsubsection{Verification of SRN Hypothesis - the $\Pi_0$ gradient flow}
We establish that the manifold $\cM_{n,M}$ and the family of associated linearized operators $\{\mbL_\vp\}_{\vp\in\cP}$ satisfy 
the hypotheses (\textbf{H0})-(\textbf{H4}). To establish (\ref{e:H0}) of (\textbf{H0}), we recall the form of the residual, (\ref{e:resid}). Since $\Pi_0$ 
annihilates constants, it follows that $\Pi_0 \lambda=0$ and
\beq
\|\cR\|_{H} =  \|\cL \cR_n \|_{H} + O(\delta^2).
\eeq
The residual term is dominated by tail-tail interactions of the adjacent pulses. For $j=1, \ldots, n-1$ we introduce the midpoints $m_j:=(p_{j}+p_{j+1})/2$ and set $m_0=0$ and $m_{n}=\dell/\eps.$ 
We partition 
$$\left[0,\dell/\eps\right]=\cup_{j=0}^{n} [m_j,m_{j+1}],$$ 
and on the interval $\cI_j:=[m_{j-1},m_{j}]$ we write 
\beq
\label{e:T_j-def}
 u_n= \phi_{h,j} + T_j,
 \eeq
where $\phi_{h,j}:=\phi_h(z-p_j)$ and the tail term $T_j:=\sum_{k\neq j} \ophi_h(z-p_k).$
Expanding the n-pulse residual on $\cI_j$ we obtain
\begin{equation}
\label{e:cR-n}
\cR_n = (\partial_z^2 -W''(\phi_{h,j})) T_j - \frac12 W'''(\phi_{h,j}) T_j^2 + \cO\left(\delta^{\frac32}\right), \hspace{0.5in} {\rm for}\, z\in\cI_j.
\end{equation}
We introduce the far-field operator $\mrL_\infty:= \partial_z^2 -\alpha_-$ and write
\begin{equation}
\label{e:cR_n}
\cR_n = \mrL_\infty T_j - (\alpha_--W''(\phi_{h,j})) T_j - \frac12 W'''(\phi_{h,j}) T_j^2 + \cO\left(\delta^{\frac32}\right), \hspace{0.5in} {\rm for}\, z\in\cI_j.
\end{equation}
Using the facts that $\mrL_\infty e^{\pm\sqrt{\alpha_-} z}=0$,  that the function $\alpha_--W''(\phi_{h,j})$ decays exponentially away from $z=p_j$, and that the
functions in $\cR_n$ are smooth with $L^2$ norms of all derivatives of the same order, it is straightforward to estimate that
\beq
 \|\cR_n\|_{H^4(\cI_j)} = O(\delta).
 \label{e:H0_verified}
 \eeq
 Summing over the intervals we obtain (\ref{e:H0}).

To establish (\textbf{H1}) we observe from (\ref{secvarI2}) and (\ref{e:H0_verified}) that we have the decomposition
\beq
\label{e: mbL-pert}
 -\Pi_0\cL\Pi_0 = -\Pi_0 \mrL^2 \Pi_0 +\cO(\delta),
 \eeq
 where the error terms are small and relatively compact as operators on $H.$ We first examine the operator $\mrL$ acting on $H^2(\R)$, 
 where it is a self-adjoint Sturm Liouville operator arising as the linearization of the pulse equation (\ref{e:freeway}) about the homoclinic pulse $
 \phi_h$. The spectrum of $\mrL$ is real and takes the form 
 $\sigma(\mrL) = [-\infty, -\alpha_-] \cup \{\lambda_r< \ldots < \lambda_2<\lambda_1=0< \lambda_0\} ,$ where the number of point spectrum, 
 $r\geq1$ is finite and depends  upon the choice of well $W$. Since $u_n$ is an $n$-pulse constructed from $n$ well-separated copies of $\phi_h$, the results of \cite{Sands} imply that the point spectrum of $\mrL_n$, the linearization of (\ref{e:freeway}) about $u_n$, is composed of 
 $n$ copies of $\sigma_p(\mrL)$, up to $O(\delta).$ That is,  to each $\lambda_k\in\sigma_p(\mrL)$, 
 there are $n$ eigenvalues 
 $\{\lambda_{k,j}\}_{j=1}^n\in \sigma_p(\mrL_n(\vp))$,  such that 
 $\max\limits_{j=1,...,n}|\lambda_{k}-\lambda_{k,j}|=\cO(\delta)$.  By standard perturbation theory,
 restricting the operator $\mrL_n$ to act on the bounded domain $H$ perturbs the point spectrum by at most $\cO(\delta)$, 
 see  \cite[Section 9.6]{KapProbook}, for a detailed discussion.  By the spectral mapping theorem, since $\mrL_n$ is self-adjoint on $H$, 
 $\sigma(\mrL_n^2)=\left\{\lambda^2 \bigl |\, \lambda\in\sigma(\mrL_n)\right\}.$ In particular we have 
 \beq
   \sigma(\mrL_n^2) \subset \left\{ \lambda_{1,1}^2\leq  \ldots \leq \lambda_{1,n}^2\right\} \cup [k_s,\infty),
   \eeq
  where $k_s:=\min\{\lambda_2^2, \alpha_-^2\}>0$ independent of $\eps$ and $\delta.$ 
  

 To localize the spectrum of $\Pi_0\mrL^2_n\Pi_0$  we introduce the bilinear form 
  \beq
  \label{def-bn}
  b_n(u,v):=((\mrL_n^2-\mu)u,v)_{L^2},
 \eeq
 constrained to act on  $u,v\in H\cap X_\cG=\{1\}^\perp.$  The constrained operator  $\Pi_0\mrL_n^2\Pi_0$ is induced by bilinear form acting
 on $H\cap X_\cG$, while $\mrL_n^2$ is induced by the form acting on all of $H$. The Rayleigh-Ritz formulation of eigenvalues  implies that the spectrum of $\Pi_0\cL_n^2\Pi_0$ is generically more positive than 
 the spectrum of $\mrL_n^2$ since the minimization in the Raleigh-Ritz formulation is taken over smaller spaces.  
 More specifically, recalling the notation $\mn(L)$ that denotes the number of negative eigenvalues of a 
  self-adjoint operator $L$, we deduce that $\mn(\Pi_0 (\mrL_n^2-\mu)\Pi_0) \leq \mn(\mrL_n^2-\mu) $ for all values of $\mu.$  
  In particular for $\mu\in (c_0\delta, k_s)$ we have
  \beq
  \label{e:n-index1}
   \mn(\Pi_0(\mrL_n^2-\mu)\Pi_0) \leq \mn(\mrL_n^2-\mu)= n.
  \eeq  
 However the projection off of the constant vector $1$, is not perturbative, our analysis requires an exact measure of the dimension of the slow space.   
 To establish that $ \mn(\Pi_0 (\mrL_n^2-\mu)\Pi_0)=n$, we show that $\Pi_0 (\mrL_n^2-\mu)\Pi_0$ is negative on the $n$-dimensional tangent space $\cT(\vp)\subset H\cap X_\cG.$
The estimates employed to establish (\textbf{H0})  verify that $\|\mrL_n^2 \frac{\partial\Phi}{\partial p_j}\|_{L^2} = O(\delta)$ for $j=1, \ldots, n$ and 
  $$\left\langle \frac{\partial\Phi}{\partial p_i},\frac{\partial\Phi}{\partial p_j}\right\rangle_{L^2} = \|\phi_h'\|_{L^2}^2 \delta_{ij} + O(\delta).$$
In particular we deduce that
$$ M_{ij}:=\left\langle  (\Pi_0 \mrL_n^2 \Pi_0 -\mu)\frac{\partial\Phi}{\partial p_i}, \frac{\partial\Phi}{\partial p_j}\right\rangle
=\left\langle  (\mrL_n^2  -\mu)\frac{\partial\Phi}{\partial p_i}, \frac{\partial\Phi}{\partial p_j}\right\rangle = \mu \delta_{ij}  \|\phi_h'\|_{L^2}^2 + \cO(\delta).$$
 For $\delta$ sufficiently small the matrix $M$ is diagonally dominant and is indeed a perturbation of the matrix $-\mu I_{n\times n}$ with $n$ 
 negative eigenvalues. We deduce that $\mn(\Pi_0\mrL_n^2\Pi_0-\mu)=n$ for  $\mu\in (c_0\delta, k_s)$, and hence $-\Pi_0\mrL_n^2\Pi_0$ 
 enjoys the slow-stable decomposition of (\textbf{H1}). This decomposition extends to $\mbL=-\Pi_0\cL\Pi_0$, modulo an $\cO(\delta)$ perturbation to $k_s$, since this operator is a self-adjoint $O(\delta)$-perturbation of $-\Pi_0\mrL^2\Pi_0$.

To establish (\textbf{H2}) we observe that for each $\vq\in\cP$ the space $Y_\vq^\perp$ is the range of the spectral 
projection associated to the stable spectrum, which in turn is contained in the the set $\{ \lambda \,\bigl |\, \Re\lambda \leq k_s\}.$ 
It follows that the resolvent $(\mbL-\lambda)^{-1}$ is uniformly bounded for these $\lambda$ as an operator on $Y_\vq'$.
The semigroup estimate (\ref{e:H2}) follows directly from application of the Gearhardt-Pr\"uss Theorem, see \cite{Gearhart} and \cite{Pruss}.    

 The verification of hypotheses (\textbf{H3}) follows from the spectral decomposition (\textbf{H1}). Indeed the spectral decomposition and the Raleygh-Ritz variational eigenvalue formulation
 implies that 
 \beq
 \|\mbL v \|_{X} \geq k_s \|v\|_{X},
 \eeq
 for all $v\in Y_\vp'. $ From a standard interpolation argument, the linear nature of the leading order fourth-derivative term in $\mbL$ 
 affords the existence of $\mu>0$, independent of $\eps$, for which
  \beq
  \label{e:X-H-coercivity}
 \|\mbL v \|_{X} \geq \mu \|v\|_{H}.
 \eeq
 We decompose the tangent-plane basis elements as
 \beq
 \frac{\partial \Phi}{\partial p_i}  = \sum\limits_{j=1}^n \beta_{ij} \psi_j  + \psi_i^\perp,
 \eeq
 where $\psi_i^\perp \in Y_\vp',$  and apply $\mbL$.   Taking the $L^2$ norm and using the triangle inequality we obtain the upper bound
 \beq \| \mbL \psi_i^\perp \|_{L^2} \leq  \left\| \sum\limits_{j=1}^n \beta_{ij} \lambda_j\psi_j  \right\|_{L^2} + \left\|\mbL \frac{\partial \Phi}{\partial p_i}\right\|_{L^2}.
 \eeq
 For each $i=1, \ldots, n$, we have $|\lambda_i|\leq c_0\delta$ while $\|\mbL \frac{\partial \Phi}{\partial p_i}\|_{L^2} =O(\delta)$; 
 we infer from the $H$-coercivity estimate that $\|\psi_j^\perp\|_{H}=O(\delta).$
 Since the matrix $\beta$ maps $\R^n$ to $\R^n$ is symmetric and maps an orthonormal basis of $Y_\vp$ asymptotically close to the asymptotically orthonormal basis of $\cT$, it is close to an orthogonal matrix. Using $\beta$ to reparameterize the pulse coordinates yields (\ref{e:H3}).
 
The hypothesis (\textbf{H4}) follows from the well-known analytic parametric dependence of the eigenvectors of an unbounded, self-adjoint operator 
with compact resolvent,  see for example \cite{KriMicRai}.

This verifies the hypotheses of Theorem\,\ref{t:KP}, in particular we deduce the reduced flow (\ref{pulseODE}) for the pulse dynamics
in the zero-mass gradient flow of the FCH energy.

\subsection{$\Pi_0$-Gradient Pulse Dynamics}

The application of Theorem\,\ref{t:KP} gives the ODE system \eqref{pulseODE} for the pulse positions. 
To simplify this flow and obtain the stability of the equispaced pulse, we first write the system mass to be in the form 
$M= n M_h+ M_1$, where $M_1\in(0,M_h)$ is $\cO(1)$, and $M_h=\int_\R (\phi_h-b_-)\,dz$ is the mass the homoclinic pulse in the
scaled variables. From Lemma\,\ref{l:App} of  the Appendix, the mass parameter $\lambda$ satisfies (\ref{e:lambda}). We recall the 
decomposition of the domain $[0,\dell/\eps]$ into the union of $\cI_j$, $j=1, \ldots, n$, and the form (\ref{e:cR-n}) of the n-pulse residual. 
For the pulses away from the boundary, that is for $i=2, \ldots, n-1,$ we have $\Pi_0 \phi_{h,i}=\cO(\delta^{\frac32})$ and we reduce
the the inner product in  \eqref{pulseODE} to the sum  
\begin{equation}
 \label{pulseODE2}
\dot{p}_i= -\frac{1}{\norm{\phi_h'}_{L^2(\R)}^2} \sum\limits_{j=1}^n \big\langle \mrL_j T_j + \frac12 W'''(\phi_{h,j})T_j^2,\partial_z\phi_{h,i}\big\rangle_{\cI_j}+\cO(\delta^{\frac32}).
\end{equation}
where we have introduced the local operator $\mrL_j := \partial_z^2 -W''(\phi_{h,j})$ considered to act on the unbounded domain. 
The function $\partial_z\phi_{h,i}$ lies in the kernel of $\mrL_i$, and for $j=i$ we determine that
\beq
\big\langle \mrL_i T_i,\partial_z\phi_{h,i}\big\rangle_{\cI_i} = - (\partial_z\phi_{h,i})(\partial_z T_i)\bigl|_{m_{i-1}}^{m_i}.
\eeq
Similarly, for the second term on the right-hand side of (\ref{pulseODE2}) we write 
$W'''(\phi_{h,i})\partial_z \phi_{h,i} = \partial_z\left( W''(\phi_{h,i})\right),$ and integrate by parts to obtain
\beq
\big\langle  \frac12 W'''(\phi_{h,i})T_i^2,\partial_z\phi_{h,i}\big\rangle_{\cI_i} = - \big\langle  T_i\partial_z T_i, W''(\phi_{h,i})\big\rangle_{\cI_i} + \frac12 W''(\phi_{h,i}) T_j^2\bigl|_{m_{i-1}}^{m_i}.
\eeq
Since $\phi_h$ tends to $b_-$ at an exponential rate, replacing $W''(\phi_{h,i})$ with is constant asymptotic value $\alpha_-$ incurs an 
$\cO(\delta^{\frac32})$ error in the integral and the boundary term, while integrating by parts on $\langle  T_i\partial_z T_i, \alpha_- \big\rangle_{\cI_i}$ cancels out the leading order boundary terms. We deduce that
\beq
\big\langle  \frac12 W'''(\phi_{h,i})T_i^2,\partial_z\phi_{h,i}\big\rangle_{\cI_i} = \cO\left(\delta^{\frac32}\right).
\eeq
For $j=i\pm 1$ the quadratic term $W'''(\phi_{h,j})T_j^2$ is uniformly $\cO(\delta^{\frac32})$ and hence negligible. 
The linear term, $\mrL_j T_j$, takes the form,
\beq
\big\langle \mrL_j T_j,\partial_z\phi_{h,i}\big\rangle_{\cI_j} = - (\partial_z\phi_{h,i})(\partial_z T_j)\bigl|_{m_{j-1}}^{m_j} + \bigl\langle T_j, (\alpha_--W''(\phi_{h,j})\partial_z \phi_{h,i}\bigl\rangle_{\cI_j}.
\eeq
The integrand in the inner product term on the right-hand side has $L^\infty$ norm $\cO(\delta^{\frac32})$ and is negligible. 
The inner product on the left-hand side is dominated by the boundary terms; recalling  the definition of $T_j$ and keeping only leading 
order terms we find
\beq
\dot{p}_i =  -\frac{ -\partial_z \phi_{h,i}\partial_z \phi_{h,i+1}\bigl|_{m_i} +\partial_z \phi_{h,i}\partial_z \phi_{h,i-1}\bigl|_{m_{i-1}}\!\!\!\!
                                                                                 -\left( \partial_z \phi_{h,i}\right)^2\bigl|_{m_{i-1}} +  \left( \partial_z \phi_{h,i}\right)^2\bigl|_{m_{i}} }
                                                                                        {\|\partial_z\phi_h\|_{L^2}} +\cO(\delta^{\frac32}).
\eeq
The pulse profiles have the far-field asymptotic form
\beq
\phi_h(z) = \phi_{\textrm{max}} e^{-\sqrt{\alpha_-} |z|},
\eeq
where the constant $\phi_{\textrm{max}}$ is determined by matching to the exact pulse shape $\phi_h$.  Since  $p_{i-1}<m_{i-1}<p_i<m_i<p_{i+1}$ it follows that
$\partial_z \phi_{h,i}(m_i) \partial_z \phi_{h,i+1}(m_i)<0$ and $\partial_z \phi_{h,i}(m_{i-1}) \partial_z \phi_{h,i-1}(m_{i-1})<0$. We conclude that
\beq
\label{e:n-pulse_dynamics}
\dot{p}_i =  -\frac{2\alpha_-\phi_{\textrm{max}}^2}{\|\partial_z\phi_h\|_{L^2}}
                   \left( e^{-\sqrt{\alpha_-}(p_{i+1}-p_i)}  -e^{-\sqrt{\alpha_-}(p_i-p_{i-1})}\right)  +\cO(\delta^{\frac32}),
\eeq
for $i=2, \ldots, n-1$. The same result for $i=1,n$ follows by replacing the boundary correction terms $E$ in (\ref{e:E-def})
with a pulse located at $p_0$ and $p_{n+1}$ given by Lemma\,\ref{l:App}. This replacement incurs a higher order error, 
and the analysis above extends to the cases $i=1, n.$

For a given $\dell$ and $n$ there is a unique equally spaced pulse configuration with $p_{i+1}-p_i= \frac{\dell}{n\eps}$ for $i=0, \ldots, n+1$.
Here we recall that the $p_0$ and $p_{n+1}$ denote the placements of shadow pulses outside the domain $[0,\dell/\eps].$ 
We conclude from (\ref{e:n-pulse_dynamics}) that if the pulses are equally separated then the pulse locations are stationary to leading order. Furthermore, the Jacobian matrix of the 
ODE system taken at the equispaced pulse locations takes the form
\begin{equation}
J=-\begin{pmatrix}
    \gamma & -\frac{\gamma}{2} & 0      &0   & \dots      & 0   \\
     -\frac{\gamma}{2} & \gamma       & -\frac{\gamma}{2}    & 0    &\dots     & 0   \\
     0 & -\frac{\gamma}{2}       & \gamma         & -\frac{\gamma}{2} &0   &\vdots                          \\
   \vdots &   0    &  \ddots        &  \ddots       &   \ddots     &0           \\
   0 & \vdots&\ddots &\ddots &  \ddots&-\frac{\gamma}{2} \\
     0 &   0      &    \dots       &  0     &        -\frac{\gamma}{2}                     & \gamma
  \end{pmatrix}
\end{equation}
where $\gamma:=\frac{2 \alpha_-\phi_\textrm{max}^2}{\|\partial_z \phi_h\|_{L^2}} e^{-\sqrt{\alpha_-}{\ell}} = \frac{2 \alpha_-\phi_\textrm{max}^2}{\|\partial_z \phi_h\|_{L^2}}\delta. $
The standard result for spectrum of tri-diagonal matrices shows that $J$ has $n$ negative eigenvalues 
\begin{equation}
\lambda_k=-\gamma\left(1+\cos\left({\frac{k}{n+1}}\right)\right)<0, \hspace{0.5in} \textrm{for}\quad k=1,\dots, n.
\end{equation}
We conclude that the equispaced pulse solution is linearly stable under the leading-order flow. Since the flow for $\vp$ is smooth, 
there exists an $\cO(\delta^{\frac12})$ neighborhood of the equispaced pulse configuration that is forward invariant under the flow.
Initial data of the system (\ref{e:FCH-flow}) corresponding to initial data $u_0$ with a decomposition (\ref{e:u0-init}) with $\|w\|_H=\cO(\delta)$
and $\vp_0$ within $\cO(\sqrt{\delta})$ of the equispaced pulse configuration will remain within $\cO(\sqrt{\delta})$ of the equispaced
pulse configuration for all time.

\subsection{EL Assumption Verification - General Gradients}
To apply Theorem\,\ref{t:KP-BFK} for the flow (\ref{G-flow}) with a general gradient $\cG\in\cC_{X_0},$ we
must verify that (\textbf{A0}) and the assumptions (\textbf{EH1}) and (\textbf{EA}) hold, and impose conditions
for which $\eta_*<\eta^*$.  From the form of (\ref{Phi-def}), and more particularly
(\ref{Phi-def_s}) it is straight forward to see that
\beq
J(\Phi) = \left\|\sum_{j=1}^n L_j (T_j+\delta P)\right\|_X^2 \leq c_0 \delta,
\eeq
for some $c_0>0$ independent of $\vp\in\cP.$ This bound is sharp since from (\ref{pulseODE}) we have the
leading order result
\beq
\partial_{p_j} J(\Phi) = \left\langle \nabla_X J(\Phi),\frac{\partial \Phi}{\partial p_j} \right\rangle_X  = \dot{p}_j + O(\delta^2).
\eeq
Introducing the equispaced $n$-pulse $\vp_{\rm eq}$ then from (\ref{e:n-pulse_dynamics}) we see that
$$ 
|\dot{\vp}| \geq d_0 \delta |\vp-\vp_{eq}|,
$$
for some $d_0>0$ independent of $\vp\in\cP$ and $\delta.$ It is trivial to show that the set of $u_0\in X_0\cap H$ with 
$$J(u_0)<\sup\limits_{\vp\in\cP}J(\Phi(\vp)) + \delta,$$
is non-empty, since this set contains the manifold $\cM_{n,M}.$  Thus we may take $\delta_0=c_0\delta$ and $\delta_1=\delta,$
for which choice we have $\eta_*=O(\sqrt{\delta})$ and this upper bound is asymptotically sharp for a set $\cP$ that is at least
$O(1)$ wide.  The assumption (\textbf{EH1}) is satisfied by construction
of $\cM_{n,M}$, while the normal coercivity assumption (\textbf{EA}) is equivalent to the argument used to establish (\ref{e:X-H-coercivity}).
 Indeed we may write $\nabla_X^2\mrI(\Phi)=\cL$ in the form
$$ \cL = \partial_z^4 + q_2(z) \partial_z^2 +q_1(z) \partial_z +q_0(z) + \alpha_-^2,$$
where $q_2, q_0\in L^2(0,\dell/\eps).$ For $\gamma_e>0$ sufficiently large we may write
$$ \cL = (\partial_z^4+\alpha_-^2+\gamma_e)\left( I+ B\right),$$
where $B:= (\partial_z^4+\alpha_-^2+\gamma_e)^{-1}\left(q_2(z) \partial_z^2 +q_1(z) \partial_z +q_0(z)\right),$ is a bounded map from $H$ into $H$
whose norm decreases to zero with increasing values of $\gamma_e$. The assumption (\textbf{EA}) follows.

We deduce that for any gradient, in particular the $H^{-1}$ gradient $\cG=-\partial_z^2$, that the manifold $\cM_{n,M}$ is 
quasi-steady under the flow (\ref{G-flow}). In particular if $u_0$ is within a $\eps$-neighborhood of $\cM$ in the $H$ norm,
and satisfies (\textbf{A0}) with $\delta_1=\delta$, then it is within an $\eta_*=O(\sqrt{\delta})$ neighborhood and will remain
there until time $\Texit$, which can be bounded from below using \cite[Theorem 2.2]{BFK-18}.

\subsection{Pulse dynamics for the $H^{-s}$ gradient flow}
We apply Theorem\,\ref{t:dyn-equiv} to (\ref{G-flow}) for a family of gradients parameterized by $s\in[0,1]$. Defining the gradients by
their inverses, we introduce the space $L_0^2(0,\dell/\eps)$ comprised of zero-mass functions and consider the operator 
$D:L^2_0(0,\dell/\eps)\mapsto H^{2}_0$ that maps $f\in L^2_0$ onto the solution 
$u$ of
\beq
\begin{aligned}
 -u_{zz} &= f \hspace{0.5in} \textrm{in}\, (0,\dell/\eps), \\
 u_z(0) &=u_z(\dell/\eps)=0,
\end{aligned}
\eeq
subject to $\Pi_0 u =u.$  The space $L_2^0$ denotes $L^2$ functions with zero-mass, on this space the operator $D$ has eigenvalues 
$\{\lambda_n=\dell^2/(\eps^2\pi^2n^2)\}_{n=1}^\infty$, which tend to zero as $n\to\infty.$ Consequently its norm is given by $\lambda_1=\dell^2/(\pi^2\eps^2)$. The operator $D^s$ denotes the $s$'th root of $D$, with the same eigenfunctions but eigenvalues defined equal to $\{\lambda_n^s\}_{n=1}^\infty.$  Correspondingly, we establish a norm-1 inverse operator by setting
$\cG=\lambda_1^{s}D^{-s}$ so that 
\beq
\label{e:cG1-def}
\cG_1:= \lambda_1^{s/2}D^{-s/2}= \frac{\dell^s}{\eps^{s} \pi^{s}} D^{-s/2},
\eeq
has smallest non-zero eigenvalue equal to $1$.  In particular, for $s=0$ we have $\cG=\cG_1=\Pi_0$ while for $s=1$ we have 
$\cG=\frac{\dell^2}{\eps^2\pi^2}D^{-1}=\frac{\dell^2}{\eps^2\pi^2}\partial_z^2$ and $\cG_1= \frac{\dell}{\eps\pi}D^{-\frac12}$. For $s=1$,
the operator $\cG$ is proportional to $\partial_z^2$, however $\cG_1$ is a positive, self-adjoint operator and is not proportional to $\partial_z$.

Theorem\,\ref{t:KP} has been established for gradient $\Pi_0$, we extend it to recover the pulse dynamics for the $H^{-s}$ gradient
flow for $s\in[0,1].$  To address the nonlinear estimate (\textbf{EH2}) we remark that for $v\in H^4$, we have the expansion,
$$ \cN_S(v) = \cG\left( W'''(\Phi) v \mrL_n v-\frac12 \mrL_n(W'''(\Phi)v^2)\right) + O(\|v\|_{H^4}^3),$$
where the operator $\mrL_n$ is defined in (\ref{e:Ln-def}).
We must establish identify a large parameter $\rho=\rho(\eps)$ for which we have the bound
\beq
\label{e:FCH-EH2}
 \|\rho\cG_1 \cN_S(\rho^{-1} \cG_1 w) \|_{H_{\cG_1}}= \|\rho\cG_1^2 \cN_S(\rho^{-1} \cG_1 w) \|_{H^4} \leq c \|\cG_1 w\|^2_{H^4},
 \eeq
for some constant $c>0$, independent of $\eps$ and $\rho$. 
The argument of the norm on the left-hand side has leading order terms
$$ \rho \cG_1^2 \cN_S(\rho^{-1} \cG_1 w)\sim (\eps^{2s}\rho)^{-1} D^{-s} \left( W'''(\Phi) (\cG_1 w)\mrL_n (\cG_1 w) -\frac12 \mrL_n(W'''(\Phi)( \cG_1 w)^2)\right).$$
Since the potential $W$ and the profile $\Phi$ are smooth, $D^{-s}$ is bounded as a map from $H^{2s}$ to $L^2$, 
$\mrL_n$ is bounded as an operator from $H^2$ into $L^2$, and the $H^k$ norm is an algebra on $\R$ for $k>1/2$, we have the 
estimate
$$ \|\rho\cG_1^2 \cN_S(\rho^{-1} \cG_1 w) \|_{H^4} \leq c \| \cG_1 w\|_{H^{2+2s}}^2,$$
so long as $\rho\geq \eps^{-2s}.$ This establishes (\ref{e:FCH-EH2}) and hence (\textbf{EH2}) for $s\in[0,1].$

To establish the bounds in (\textbf{EH3}), we recall that $\nabla_x\cJ(\Phi(\vp))=\cR(\vp)$ and return to the identities (\ref{e:resid}) and (\ref{e:cR_n}).
Applying the $H_{\cG_1}$ norm to $\cG_1\cR$ and using the scaling (\ref{e:cG1-def}), we find that (\ref{e:EH3-1}) holds with $\rho=\eps^{-s}.$ 
If $u\in\cT_\vp$, then up to exponentially small terms, $u$ is a linear combination of translates of $\phi_h'$ and (\ref{e:EH3-2}) holds with 
$\rho=\eps^{-s}$. Since $\delta=e^{-\sqrt{\alpha_-}\ell}$ and $\ell \gg |\ln\eps|$ it follows that $\delta \ll\eps^p$ for any $p>0$ and in particular
$\rho^2\delta = \eps^{-2s}\delta\ll 1$ for any choice of $s\in [0,1].$ This establishes Theorem\,\ref{t:dyn-equiv} for this range of gradients.

To interpret the scale of the reduced flow (\ref{pulseODE}) we first must identify the proper reparameterization the the pulse locations 
$\tvp=\tvp(\vp)$ for which (\textbf{H3}) holds with eigenfunctions $\Psi_i$ given by (\ref{e:Psi-decomp}) and $\Phi$ replaced with $\cG_1^{-1}\Phi$. 
This requires the normalization  $\|\cG_1^{-1} \partial_{\tvp_i} \Phi\|_{L^2}=1$ for $i=1, \ldots, n$, and can be achieved via the 
linear transformation $\tvp=\alpha\vp+\vp_*$ where $\vp_*$ is a fixed vector in $\R^n$ and the scaling constant
$$\alpha(s):= \|\cG_1^{-1} \Pi_0 \phi_h^\prime\|_{L^2} = \lambda_1^{-s/2}\|D^{s/2}\Pi_0\phi_h'\|_{L^2}.$$
It is straightforward to calculate that, up to exponentially small terms,  $\alpha(0)=\|\phi_h'\|_{L^2(\R)}$ and 
$\alpha(1)= \left(\frac{\eps\pi}{d}\right)^s\|\Pi_0\phi_h\|_{L^2(\R)}.$ Moreover $\alpha$ is a strictly decreasing function of $s$ as
all the eigenvalues of $\cG_1^{-1}$ are less than or equal to one, hence its norm decreases with growing $s$.
Changing variables from $\tvp$ to $\vp$ in (\ref{PulseODE-cG1}) we find
\beq
\dot{p_i} = \frac{1}{\alpha^{2}(s)} \left\langle \cR(\vp), \frac{\partial \Phi}{\partial p_i} \right\rangle + O\left(\alpha^{-1}\delta_\cG^{1+r},\alpha^{-1}\delta_\cG^2\right), \hspace{0.5in} \textrm{for}\, i=1, \ldots, n.
\eeq
The inner-product on the right-hand side equals the leading order term on the right-hand side of (\ref{e:n-pulse_dynamics}). 
This demonstrates that the impact of the change of gradient on the leading-order pulse dynamics amounts to a rescaling of their velocity.
\section{Acknowledgments}
The second author acknowledges support from the National Science Foundation through grant DMS-1813203.
\section{Appendix}
\small
We consider a manifold of mass $M$ as given in Section\,\ref{s:FCH} and satisfying the zero first and third derivative boundary
conditions on $[0,\dell/\eps]$. Then the ansatz $\Phi$ defined in (\ref{Phi-def}) satisfies the following estimates.
\begin{lem}
\label{l:App}
The internal parameters are given by
\beq e_0 = \sqrt{\alpha_-} \frac{d_3 -\alpha_- d_1}{d_3-3\alpha_- d_1},\hspace{0.5in}
\textrm{and}\hspace{0.5in} p_0 = -p_1 + O(\delta),
\eeq
where we have introduced $d_1=u_n'(0)+\lambda \mB_{2,n}'(0)$ and $d_3=u_n'''(0)+\lambda \mB_{2,n}'''(0)$.
Similar relations hold for $e_{n+1}$ and $p_{n+1}$. 
If $n\eps \ll 1$ and $M_1\gg \delta$ then the Lagrange multiplier $\lambda$ satisfies
%
In particular 
\beq
\label{e:lambda}
\lambda =  \eps \frac{M_1}{\dell\mB_{2,\infty}+\eps nM_\omB} +O\left(\eps \delta\right),
\eeq
and in particular $\partial_{p_i} \lambda = O\left(\eps \delta\right).$
\end{lem}
\begin{proof}
The results on the parameters $e_0$ and $p_0$ follow from a simple calculation from the form of $\Phi$ given in (\ref{Phi-def}). For the
mass we calculate the  leading order asymptotic
\beq 
\int_0^{\dell/\eps} (\Phi-b_-)\, dz = n M_h +\lambda \left( \frac{\dell}{\eps} \mB_{2,\infty} +n M_\omB\right) - \frac{e^{\sqrt{\alpha_-}p_0}+e^{-\sqrt{\alpha_-}(\dell/\eps-p_{n+1})}}{\sqrt{\alpha_-}}  +O(\eps\delta^2).
\eeq
The results follow from the assumption on the size of the mass $M$.
\end{proof}
\normalsize
\bibliographystyle{plain}

\end{document}